\documentclass[a4paper,10pt,reqno]{amsart}

\usepackage{amsthm,amsmath,amssymb}
\usepackage[foot]{amsaddr} 
\usepackage{amsfonts}
\usepackage[T1]{fontenc}
\usepackage{enumitem}
\usepackage{caption}
\usepackage{graphicx}
\usepackage[breaklinks]{hyperref} 
\usepackage{comment}
\usepackage{pgf,tikz}
\usepackage{ifthen}
\usepackage{color}
\usetikzlibrary{arrows.meta,decorations.pathmorphing,decorations.pathreplacing,calc,intersections}
\usepackage{pgfplots}
\usepgfplotslibrary{fillbetween}
\tikzset{every picture/.style={line width=0.55pt}}%
\tikzset{>={Classical TikZ Rightarrow[scale=1.66]}}
\tikzset{
  show curve controls/.style={
	postaction={
	  decoration={
		show path construction,
		curveto code={
		  \draw [blue,opacity=0.85, line width=0.2pt] 
			(\tikzinputsegmentfirst) -- (\tikzinputsegmentsupporta)
			(\tikzinputsegmentlast) -- (\tikzinputsegmentsupportb);
		  \fill [red, opacity=0.45] 
			(\tikzinputsegmentsupporta) circle [radius=1pt]
			(\tikzinputsegmentsupportb) circle [radius=1pt];
		}
	  },
	  decorate
}}}
\definecolor{linkcolor}{rgb}{0,0,0.675}%
\hypersetup{colorlinks=true,linkcolor=linkcolor,citecolor=linkcolor,urlcolor=linkcolor,bookmarksdepth=2}%
\newtheorem{theorem}{Theorem}[section]
\newtheorem{lemma}[theorem]{Lemma}
\newtheorem{corollary}[theorem]{Corollary}
\newtheorem{proposition}[theorem]{Proposition}
\theoremstyle{definition}
\newtheorem{remark}[theorem]{Remark}
\newtheorem{definition}[theorem]{Definition}
\newtheorem{example}[theorem]{Example}
\newtheorem{notation}[theorem]{Notation}
\newtheorem{algorithm}[theorem]{Algorithm}
\setitemize{topsep=0pt plus 3pt,itemsep=0pt,leftmargin=25pt,listparindent=\parindent}
\setenumerate{topsep=0pt plus 3pt,itemsep=0pt,leftmargin=25pt,listparindent=\parindent}
\setitemize{topsep=0pt plus 3pt,itemsep=0pt,leftmargin=25pt,listparindent=\parindent}
\setenumerate[1]{label=\normalfont(\arabic{enumi})}
\setenumerate[2]{label=\normalfont(\alph{enumii}),leftmargin=25pt}
\setdescription{font=\sffamily, itemsep=0pt,leftmargin=8pt,listparindent=\parindent}

\makeatletter\@addtoreset{equation}{section}\makeatother

\newcommand*{\Spec}{\operatorname{Spec}}
\newcommand*{\Tors}{\operatorname{Tors}}
\newcommand*{\Pic}{\operatorname{Pic}}

\newcommand*{\e}{\op{\sf e}}

\newcommand*{\xto}[1]{\xrightarrow{\,#1\,}}

\makeatletter
\newcommand*{\da@rightarrow}{\mathchar"0\hexnumber@\symAMSa 4B }
\newcommand*{\da@leftarrow}{\mathchar"0\hexnumber@\symAMSa 4C }
\newcommand*{\xdashrightarrow}[2][]{%
  \mathrel{%
	\mathpalette{\da@xarrow{#1}{#2}{}\da@rightarrow{\,}{}}{}%
  }%
}
\newcommand{\xdashleftarrow}[2][]{%
  \mathrel{%
	\mathpalette{\da@xarrow{#1}{#2}\da@leftarrow{}{}{\,}}{}%
  }%
}
\newcommand*{\da@xarrow}[7]{%
  \sbox0{$\ifx#7\scriptstyle\scriptscriptstyle\else\scriptstyle\fi#5#1#6\m@th$}%
  \sbox2{$\ifx#7\scriptstyle\scriptscriptstyle\else\scriptstyle\fi#5#2#6\m@th$}%
  \sbox4{$#7\dabar@\m@th$}%
  \dimen@=\wd0 %
  \ifdim\wd2 >\dimen@
	\dimen@=\wd2 %
  \fi
  \count@=2 %
  \def\da@bars{\dabar@\dabar@}%
  \@whiledim\count@\wd4<\dimen@\do{%
	\advance\count@\@ne
	\expandafter\def\expandafter\da@bars\expandafter{%
	  \da@bars
	  \dabar@ 
	}%
  }%
  \mathrel{#3}%
  \mathrel{%
	\mathop{\da@bars}\limits
	\ifx\\#1\\%
	\else
	  _{\copy0}%
	\fi
	\ifx\\#2\\%
	\else
	  ^{\copy2}%
	\fi
  }%
  \mathrel{#4}%
}
\makeatother

\renewcommand*{\mod}{\operatorname{mod\,}}
\newcommand*{\ZZ}{\mathbb{Z}}

\newcommand*{\CC}{\mathbb{C}}
\newcommand*{\sZ}{{\mathsf{Z}}}

\let\op\operatorname
\author{Yonghwa Cho}
\address{Department of Mathematics, Gyeongsang National University, 501 Jinju-daero, Jinju 52828, Gyeongsangnam-do, Republic of Korea}
\email{yhcho@gnu.ac.kr}

\title{Counting $h^0(D)$ on primary Burniat surfaces}

\thanks{This work(GNU-2024-240034) was supported by the research grant of the new professor of the Gyeongsang National University in 2024.} 


\begin{document}
\setlength{\intextsep}{0.5\baselineskip plus 2pt minus 2pt}
\begin{abstract}
We study the cohomology of divisors on a Burniat surface $X$ with $K_X^2=6$. We provide an algorithm for computing the cohomology groups of arbitrary divisors on $X$. As an application, we prove that there are no Ulrich line bundles\,(with respect to an arbitrary polarization), and that there exists an Ulrich vector bundle of rank $2$ with respect to $3K_X$. The existence of Ulrich vector bundle of rank $2$ was previously established by Casnati, but our construction yields one that cannot be obtained by his method.
\end{abstract}
\maketitle
\section{Introduction}

A Burniat surface is a surface of general type with $p_g=q=0$ and $2 \leq K^2 \leq 6$, which can be obtained as a bidouble cover of a del Pezzo surface of degree $6$. The configuration of branch curves is called the Burniat configuration, which plays a crucial role in the study of divisors on these surfaces. In this article, we focus on the primary Burniat surfaces, namely, those with $K^2=6$. See \cite{BauerCatanese:Burniat1} for historical background and related results on Burniat surfaces.

In \cite{AlexeevOrlov:Burniat}, Alexeev and Orlov presented an exceptional collection of maximal length on the primary Burniat surface by analyzing the ramification curves over the Burniat configuration. The study was further extended in \cite{Alexeev:Burniat}, which presents a precise description of the semigroup of effective divisors.

The main goal of this article is to extend further the result of \cite{Alexeev:Burniat} for the primary case. We not only focus on the effectiveness of the divisors $D$, but also provide the systematic method for computing $H^p(D)$. It turns out that it is not  straightforward to find a closed formula for $H^p(D)$. Instead, we provide an algorithm that computes $H^p(D)$ completely, with a reasonable efficiency. This is the main result of the paper. 
\begin{theorem}\label{thm: Main Thm}
    Let $X$ be a primary Burniat surface with the canonical divisor $K_X$. Given $D \in \Pic X$, the Algorithm~\ref{Main Algorithm} computes the dimension of $H^p(D)$ for each $p=0,1,2$, and its time complexity is $O(d)$ where $d = (D. K_X)$.
\end{theorem}
However, as substantial preparation is required, a precise formulation of the algorithm is postponed to Section~\ref{sec: Main Algorithm}.

Our original motivation is to see whether $X$ supports an Ulrich line bundle. For a smooth projective variety $(Z, \mathcal O_Z(1))$, an Ulrich bundle over $Z$ is a locally free sheaf $\mathcal E$ satisfying
\[
    H^p( Z, \mathcal E \otimes \mathcal O_Z(-i))=0,\ \text{for each } p\text{ and }i=1,2,\dots,\dim Z.
\]

We refer the reader to the foundational work \cite{EisenbudSchreyer:Urlich}, which introduced Ulrich modules into algebraic geometry. For a concise and accessible introduction to Ulrich bundles, see \cite{Beauville:UlrichIntro}.

In \cite{Casnati:SpecialUlrich}, Casnati proved that if $(S,\mathcal O_S(1))$ is a projective surface with $p_g(S)=q(S)=0$ and $h^1(\mathcal O_S(1))=0$, then $S$ admits an Ulrich bundle of rank $2$. The theorem applies to the Burniat surface $X$, hence it is natural to ask whether $X$ admits an Ulrich line bundle. Our next theorem gives a negative answer to this question.

\begin{theorem}[see Theorem~\ref{thm: no Ulrich line bundles}]
    There exists no Ulrich line bundle over a primary Burniat surface with respect to any polarization.
\end{theorem}

We also present an Ulrich bundle of rank $2$ over $(X, \mathcal O_X(3K_X))$ which cannot be obtained by the method of Casnati\,(see Theorem~\ref{thm: Ulrich of rank 2} and Remark~\ref{rmk: why differ from Casnati's}).

\medskip
\subsection*{Notations and Conventions}\ 
    \begin{enumerate}
        \item Throughout the paper, $X$ is a fixed primary Burniat surface over $\CC$.
        \item Unless stated otherwise, a divisor $D$ on $X$ will always refer to the linear equivalence class that contains $D$. Accordingly, if $D_1$ and $D_2$ are divisors, then $D_1 = D_2$ implies that $D_1$ and $D_2$ are linearly equivalent.

        \item Let $\sim_{\sf num}$ be the numerical equivalence relation on $\Pic X$. The corresponding equivalence class of $D$ will be denoted by $[D]$. 

        \item The number of effective divisors in $[D]$ will be referred to as the \emph{{\sf e}-number $\e([D])$ of $[D]$}.
        
    \end{enumerate}

\bigskip
\section{Burniat configuration}\label{sec: Burniat config}
To fix notation, we summarize \cite[Section 2]{AlexeevOrlov:Burniat} on the configuration of curves on a Burniat surface. Let $\bar X$ be a del Pezzo surface of degree $6$. 
The surface $\bar X$ contains a hexagon of $(-1)$-curves, denoted cyclically by $\bar A_0$, $\bar C_3$, $\bar B_0$, $\bar A_3$, $\bar C_0$, $\bar B_3$. The sum of two consecutive edges forms a pencil of rational curves in $\bar X$. We choose two general members in each pencil:
\[
    \bar A_1,\bar A_2 \in \bigl\lvert \bar C_0 + \bar A_3 \bigr\rvert,\qquad
    \bar B_1,\bar B_2 \in \bigl\lvert \bar A_0 + \bar B_3 \bigr\rvert,\qquad
    \bar C_1,\bar C_2 \in \bigl\lvert \bar B_0 + \bar C_3 \bigr\rvert.
\]
Let $\bar R = \sum_{\sZ \in \{A,B,C\}}\sum_{i=0}^3 \bar \sZ_i$. It is straightforward to check that for each $\sZ \in \{A,B,C\}$, $2 L_{\sZ} := \bar R - \sum_{i=0}^3 \bar \sZ_i$ is $2$-divisible in $\Pic \bar X$. We define an $\mathcal O_{\bar X}$-algebra
\[
    \mathcal A := \mathcal O_{\bar X} \oplus \mathcal O_{\bar X}(-L_A) \oplus \mathcal O_{\bar X}(-L_B) \oplus \mathcal O_{\bar X}(-L_C),
\]
where the multiplication structure is given by the section corresponding to $\sum_{i=0}^3 \bar C_i \in \bigl\lvert L_A+L_B - L_C \bigr\rvert$, etc. The Burniat surface $X$ is defined as $X := \Spec_{\bar X} \mathcal A$. The structure morphism $\pi \colon X \to \bar X$ is the branched $(\ZZ/2)^2$-cover.
\begin{figure}[h!]
    \centering
    \begin{tikzpicture}
        \pgfmathsetmacro{\hexradius}{2.5}
        \foreach \i in {0,1,2,3,4,5}{
            \coordinate (v\i) at (180-60*\i:\hexradius);
        }
        \foreach \i/\s in {0/A_0,1/C_3,2/B_0,3/A_3,4/C_0,5/B_3}{
            \pgfmathparse{Mod(\i+1,6)}	\pgfmathtruncatemacro{\j}{\pgfmathresult}
            \draw[-] (v\i) -- (v\j);
            \draw node (\s) at ( $(0,0)!1.2! ($(v\i)!0.5!(v\j)$) $ ) {$\s$};
        }
        \foreach \i/\s in {0/C,1/B,2/A}{
            \pgfmathparse{Mod(\i+1,6)}	\pgfmathtruncatemacro{\ni}{\pgfmathresult}
            \pgfmathparse{Mod(\i+3,6)}	\pgfmathtruncatemacro{\j}{\pgfmathresult}
            \pgfmathparse{Mod(\j+1,6)}	\pgfmathtruncatemacro{\nj}{\pgfmathresult}
            
            \pgfmathsetmacro{\a}{0.33}\pgfmathsetmacro{\b}{1-\a}
            \coordinate (w\i) at ( $(v\i)!\a!(v\ni) $ );
            \coordinate (w\j) at ( $(v\j)!\b!(v\nj) $ );
            \coordinate (x\i) at ( $(v\i)!\b!(v\ni) $ );
            \coordinate (x\j) at ( $(v\j)!\a!(v\nj) $ );

            \pgfmathparse{Mod(\i,2)+1}	\pgfmathtruncatemacro{\k}{\pgfmathresult}
            \pgfmathparse{Mod(\i+1,2)+1}	\pgfmathtruncatemacro{\l}{\pgfmathresult}
            \draw[-] (w\i)--(w\j);
            \draw node[shape=circle, fill=white, inner sep=1.5pt] at ( $(w\i)!0.175!(w\j)$ ) {$\s_\k$};
            \draw[-] (x\i)--(x\j);
            \draw node[shape=circle, fill=white, inner sep=1.5pt] at ( $(x\i)!0.825!(x\j)$ ) {$\s_\l$};
        }
    \end{tikzpicture}
    \caption{Burniat configuration}\label{fig: BurniatConfig}
\end{figure}
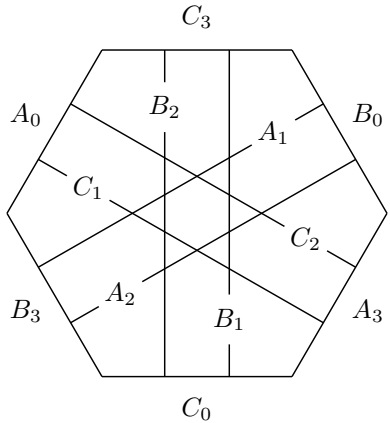
        
Each branch curve has ramification index $2$, thus $\pi^* \bar \sZ_i = 2 \sZ_i$ where $\sZ \in \{A,B,C\}$, $i=0,1,2,3$, and $\sZ_i$ is the ramification curve over $\bar \sZ_i$.\,(See Figure~\ref{fig: BurniatConfig}) For $i=0,3$, $\sZ_i \to \bar \sZ_i$ is a double cover branched at four points; for instance, $A_0$ is branched over the intersections with $B_3, C_1, C_2, C_3$. Hence $\sZ_i$ is an elliptic curve, and since $(2 \sZ_i)^2 = (\pi^* \bar \sZ_i)^2 = -4$, we have $\sZ_i^2 = -1$. If $i=1,2$, then $\sZ_i$ is a double cover branched at six points. So, $g(\sZ_i)=2$ and $\sZ_i^2 = 0$.

To equip $A_0$ with a group structure, we set the identity to be $A_0^{00} := B_3 \cap A_0$. This provides a group isomorphism
\[
    e_{A_0} \colon A_0 \to \Pic^0 A_0,\qquad P \mapsto P-A_0^{00}.
\]
Since 
\[
    2(B_3+C_0),\ 2C_1,\ 2C_2,\ 2(B_0+C_3) \in \bigl\lvert \pi^* (\bar B_3 + \bar C_0 ) \bigr\rvert,
\]
we see that $C_1\cap A_0,\, C_2\cap A_0,\, C_3 \cap A_0$ are $2$-torsion points in $A_0$. We denote those points as follows:
\[
    A_0^{01}:= C_1 \cap A_0,\quad A_0^{11}:= C_2\cap A_0,\quad A_0^{10}:= C_3 \cap A_0.
\]
In this way, we get a natural identification $ (\epsilon_1, \epsilon_2) \mapsto A_0^{\epsilon_1\epsilon_2}$ between $(\ZZ/2)^2$ and $A_0[2]$, the group of $2$-torsion points.

By setting $C_3^{00} := A_0 \cap C_3$, $B_0^{00} := C_3 \cap B_0$, and so on, we get the corresponding isomorphisms $e_{\sZ_i}$ and $(\ZZ/2)^2 \xto\sim \sZ_i[2]$ for $i=0,3$.
\begin{theorem}[see {\cite[Theorem~1]{Alexeev:Burniat} and \cite[Theorem~4.1]{AlexeevOrlov:Burniat}}]\ \label{thm: PicX and EffX}
    \begin{enumerate}
        \item  The map $\varphi \colon \Pic X \to \ZZ \oplus \Pic A_0 \oplus \Pic B_0 \oplus \Pic C_0$ 
        \[
            D \longmapsto \Bigl( (D. K_X) ,\ D\big\vert_{A_0},\ D\big\vert_{B_0},\ D\big\vert_{C_0} \Bigr) 
        \]
        is injective and the image is precisely the subgroup consisting of $(d,\alpha,\beta,\gamma)$ subject to the following conditions:
        \begin{enumerate}
            \item $d + \deg\alpha + \deg\beta + \deg \gamma$ is divisible by $3$;
            \item $ \alpha - (\deg \alpha) A_0^{00} $ is $2$-torsion.
        \end{enumerate}
        Furthermore, the curves $\sZ_i$ with $\sZ \in \{A,B,C\}$ and $i\in \{0,1,2,3\}$ generate $\Pic X$.
        \item The semigroup of effective divisors is generated by $\{\sZ_i\}$. In other words, every effective divisor is linearly equivalent to $\sum_{\sZ \in \{A,B,C\}}\sum_{i=0}^3 z_i\sZ_i$ for some nonnegative integers $\{z_i\}$.
    \end{enumerate}
\end{theorem}
Given $D \in \Pic X$, the $\Pic A_0$-component of $\varphi(D)$ can be rewritten as the pair $(a_0, \tau) \in \ZZ \oplus A_0[2]$, where $a_0 = (D. A_0)$ and $\tau = e_{A_0}^{-1}\bigl( D\big\vert_{A_0} - a_0 A_0^{00} \bigr)$. Using the identification $(\ZZ/2)^2 \xto\sim A_0[2]$ described above, we may regard $\Pic A_0$-component as an element of $\ZZ \oplus (\ZZ/2)^2$. By doing this for $B_0$ and $C_0$, we may identify $\Pic X$ as a subgroup of index $3$ in $\ZZ \oplus (\ZZ \oplus (\ZZ/2)^2)^3$.
\begin{definition}\label{def: SymCoord}
    Given a divisor $D \in \Pic X$, we call $\varphi(D) \in \ZZ \oplus (\ZZ \oplus (\ZZ/2)^2)^3$ the \emph{truncated symmetric coordinates} of $D$. The \emph{symmetric coordinates} of $D$ are augmented data $\varphi(D) \in \ZZ \oplus (\ZZ \oplus (\ZZ/2)^2)^3 \oplus (\ZZ \oplus (\ZZ/2)^2)^3$ whose additional coordinates record $( D\big\vert_{A_3}, D\big\vert_{B_3}, D\big\vert_{C_3} )$.
\end{definition}
Compared with the truncated ones, the symmetric coordinates do not provide any extra information about $D$\,(cf. \cite[Lemma~3]{AlexeevOrlov:Burniat}). Thus, the use of symmetric coordinates is merely a matter of convenience. Table~\ref{table: symCoords} displays the symmetric coordinates of $\{\sZ_i: \sZ\in\{A,B,C\},\ i=0,1,2,3\}$ and $K_X$.
\begin{table}[h!]
\[
    \begin{array}{ r@{\quad}| c | r r r@{\ \ \,} | r r r }
    & (D.K_X) & \multicolumn{1}{c@{}}{\Pic A_0} & \multicolumn{1}{c@{}}{\Pic B_0} & \multicolumn{1}{c@{}|}{\Pic C_0\ } & \multicolumn{1}{c@{}}{\Pic A_3} & \multicolumn{1}{c@{}}{\Pic B_3} & \multicolumn{1}{c@{}}{\Pic C_3} \\ \hline
    A_0 & 1 & -1\ 00 & 0\ 00 & 0\ 00 & 0\ 00 & 1\ 10 & 1\ 00 \\
    B_0 & 1 & 0\ 00 & -1\ 00 & 0\ 00 & 1\ 00 & 0\ 00 & 1\ 10 \\
    C_0 & 1 & 0\ 00 & 0\ 00 & -1\ 00 & 1\ 10 & 1\ 00 & 0\ 00 \\
    A_3 & 1 & 0\ 00 & 1\ 10 & 1\ 00 & -1\ 00 & 0\ 00 & 0\ 00 \\
    B_3 & 1 & 1\ 00 & 0\ 00 & 1\ 10 & 0\ 00 & -1\ 00 & 0\ 00 \\
    C_3 & 1 & 1\ 10 & 1\ 00 & 0\ 00 & 0\ 00 & 0\ 00 & -1\ 00 \\ \hline
    A_1 & 2 & 0\ 00 & 1\ 01 & 0\ 00 & 0\ 00 & 1\ 11 & 0\ 00 \\
    A_2 & 2 & 0\ 00 & 1\ 11 & 0\ 00 & 0\ 00 & 1\ 01 & 0\ 00 \\
    B_1 & 2 & 0\ 00 & 0\ 00 & 1\ 01 & 0\ 00 & 0\ 00 & 1\ 11 \\
    B_2 & 2 & 0\ 00 & 0\ 00 & 1\ 11 & 0\ 00 & 0\ 00 & 1\ 01 \\
    C_1 & 2 & 1\ 01 & 0\ 00 & 0\ 00 & 1\ 11 & 0\ 00 & 0\ 00 \\
    C_2 & 2 & 1\ 11 & 0\ 00 & 0\ 00 & 1\ 01 & 0\ 00 & 0\ 00 \\ \hline
    K_X & 6 & 1\ 00 & 1\ 00 & 1\ 00 & 1\ 00 & 1\ 00 & 1\ 00
    \end{array}
\]
\caption{Symmetric coordinates of $\{\sZ_i\}$ and $K_X$}\label{table: symCoords}
\end{table}
\begin{remark}\label{rmk: Some linEquiv}
    Using (truncated) symmetric coordinates, we may compare divisors and determine whether they are linearly equivalent. We list several useful linear equivalences.
    \begin{enumerate}
        \item $2(C_0+A_3) = 2A_1 = 2A_2 = 2(C_3+A_0)$
        \item $(A_0+B_3) + A_1 + B_2 = (B_0+A_3) + A_2 + B_1$
        \item $(A_0+B_3) + A_1 + B_1 = (B_0+A_3) + A_2 + B_2$
    \end{enumerate}
    The first one is obvious as they all belong to $\lvert \pi^*\bar C_1 \rvert$. The second one is less obvious, but can be verified directly by the symmetric coordinates. The last one is obtained by adding $B_1-B_2 = B_2-B_1$ to both sides of (2).
\end{remark}
\begin{remark}\label{rmk: symmetry}
    There are obvious symmetries between the generators $\{ \sZ_i\}$ of $\Pic X$; one is the cyclic permutations $(A \mapsto B \mapsto C \mapsto A)$ of letters, and the other is the flip $(i \leftrightarrow 3-i)$ of indices. We emphasize that it is not allowed to swap the letters. For instance, we will see that $h^2(A_0-B_0) \neq h^2(B_0-A_0)$.
\end{remark}
\bigskip
\section{Effective divisors and numerical equivalences}\label{sec: k-effective test}
\subsection{Effectiveness test}\label{subsec: effectiveness test}
By virtue of Theorem~\ref{thm: PicX and EffX} and Table~\ref{table: symCoords}, we may determine whether an arbitrary divisor (given in terms of symmetric coordinates) is effective or not. Let $D = \sum_{\sZ,i} z_i \sZ_i$.

\begin{description}		
    \item[Step~1] We first assume that $z_i=0$ for $i=1,2$ and $z_3 \equiv 0\ (\mod 2)$. In this case the truncated symmetric coordinates of $D$ have the following form:
\[
    \begin{array}{ c | r r r }
        \Bigl( d & \bigl(a\ 00) & \bigl(b\ 00) & \bigl(c\ 00)\Bigr)
    \end{array},
\]
where
\begin{equation}\label{eq: special effectiveness}
    \left\{
        \begin{array}{r@{}l}
            d&{}= a_0+b_0+c_0+a_3+b_3+c_3 \\
            a&{}= b_3+c_3-a_0 \\
            b&{}= c_3+a_3-b_0 \\
            c&{}= a_3+b_3-c_0
        \end{array}
    \right. \ \iff \ 
    \left\{
        \begin{array}{r@{}l@{}l}
            \ell:={}& \frac13(d+a+b+c)& \\
            a_0 ={}& b_3+c_3 - a& \\
            b_0 ={}& -b_3 -b + \ell& \\
            c_0 ={}& -c_3 - c + \ell& \\
            a_3 ={}& -b_3 -c_3 + \ell&.
        \end{array}
    \right.
\end{equation}
Hence, $D$ is effective if and only if there exist $b_3,c_3 \in 2\,{\ZZ}_{\geq 0}$ such that $a_0,b_0,c_0,a_3 \in \ZZ_{\geq 0}$.
\item[Step~2] Now we consider the general case. By Remark~\ref{rmk: Some linEquiv}, we may assume $z_1, z_2 \in \{0,1\}$ for $z \in \{a,b,c\}$. In Table~\ref{table: symCoords}, we observe that only $C_1,C_2,C_3$ affect $A_0[2]$-component in the symmetric coordinates. So, $A_0[2]$-component gives two possibilities on the pair $(c_1,c_2)$.
\[
    \begin{array}{c|c }
        A_0[2] & \text{Candidate subdivisors of }D  \\ \hline
        00 & 0\text{\ or\ } C_1+C_2+C_3 \\
        01 & C_1 \text{\ or\ }C_2+C_3 \\
        11 & C_2 \text{\ or\ }C_1+C_3 \\
        10 & C_3 \text{\ or\ }C_1+C_2 \\
    \end{array}
\]
Considering $B_0[2]$ and $C_0[2]$-components, we obtain $2\times 2\times 2 = 8$ candidates. We subtract each candidate from $D$, and return to {\sf Step 1}. If none of the candidates yields an effective divisor, then $D$ is not effective. Otherwise, $D$ is effective.
\end{description}
\medskip
\begin{example}
    Let $D$ be the divisor whose truncated symmetric coordinates are
    \[
        \begin{array}{ c | r r r }
            \Bigl( 7 & \bigl(1\ 10) & \bigl(2\ 01) & \bigl(2\ 11)\Bigr)
        \end{array}.
    \]
    Since $A_0[2]\oplus B_0[2] \oplus C_0[2]$-components are $(10\ 01\ 11)$, we may choose, for instance, $C_3 + A_1 + B_2$ as a candidate subdivisor of $D$. Then, $D - (C_3+A_1+B_2)$ has symmetric coordinates
    \[
        \begin{array}{ c | r r r }
            \Bigl( 2 & \bigl(0\ 00) & \bigl(0\ 00) & \bigl(1\ 00)\Bigr)
        \end{array}.
    \]
    Substituting $(d,a,b,c)=(2,0,0,1)$ to (\ref{eq: special effectiveness}), we find that no desirable $\{b_3,c_3\}$ exists, thus $D-(C_3+A_1+B_2)$ is not effective. However, if we replace $B_2$ by $B_1+B_3$, then 
    \[
        D - (C_3+A_1+B_1+B_3) = \bigl( 1\ \mid\ (-1\ 00)\ (0\ 00)\ (0\ 00) \bigr),
    \]
    thus $D-(C_3+A_1+B_1+B_3) = A_0$, and $D$ is effective. Another candidate $C_3+(A_2+A_3)+B_2$ also shows that $D$ is effective.
\end{example}
\medskip
\subsection{Numerical equivalences and ${\sf e}$-numbers}
The effectiveness test shows its advantage when it is implemented in a computer algebra system. From a theorist's perspective, however, the test is far less useful, as it does not provide a uniform criterion (such as inequalities in terms coordinates) for effectiveness. It appears challenging to obtain a direct criterion for effectiveness using symmetric coordinates. To get a handy criterion, we study the numerical equivalences instead.
\begin{definition}\label{def: numClass and k-eff}
    Let $D$ be a divisor with the truncated symmetric coordinates
    \[
        \begin{array}{ c | r r r }
            \Bigl( d & \bigl(a\ \text{\textasteriskcentered\textasteriskcentered}) & \bigl(b\ \text{\textasteriskcentered\textasteriskcentered}) & \bigl(c\ \text{\textasteriskcentered\textasteriskcentered}) \Bigr)
        \end{array}.
    \]
    The numerical equivalence class of $D$ is denoted by $[d;a,b,c]$.
\end{definition}
Recall that the $\sf e$-number $\e([D])$ is the number of effective divisors in $[D]$. Since the torsion subgroup of $\Pic X$ has order $64$, the $\sf e$-number ranges from $0$ to $64$. This section focuses on characterizing when the $\sf e$-number is $64$.
\begin{lemma}
    Let $D$ be a non-principal divisor. If $D$ is effective, then so is $D+K_X$.
\end{lemma}
\begin{proof}
    We proceed by induction on $d := (D. K_X)$.
    \begin{description}
        \item[Step 1] If $d=1$, then $D \in \{\sZ_i\}_{i=0,3}$. By comparing symmetric coordinates, we may see
        \[
            A_0 + K_X = A_1+A_2+A_3+2B_0.
        \]
        By symmetry\,(Remark~\ref{rmk: symmetry}), this completes the case $d=1$.
        \item[Step 2] Assume $d=2$. If $D$ is reducible, then it is a sum of two curves in $\{\sZ_i\}_{i=0,3}$, hence we are done by {\sf Step 1}. If it is irreducible, then $D \in \{\sZ_i\}_{i=1,2}$. Then, by
        \[
            A_1 + K_X = A_0 + A_2 + A_3 + 2(B_3+C_0)
        \]
        and symmetry, $D+K_X$ is effective.
        \item[Step 3] If $d \geq 3$, then by Theorem~\ref{thm: PicX and EffX}, $D-\sZ_i$ is effective for some $\sZ$ and $i$. By induction hypothesis, $(D-\sZ_i+K_X)$ is effective, thus $D = (D-\sZ_i+K_X)+\sZ_i$ is effective.\qedhere
    \end{description}
\end{proof}
Let $\tau \in \Tors \Pic X$ be a nonzero torsion. Since $\chi(\mathcal O_X) = 1$ and $h^0(\tau) = 0$, we have $h^2(\tau) = h^0(K_X+\tau) > 0$. Thus, the above lemma implies the following.

\begin{corollary}
    Assume $[D] \neq [0;0,0,0]$. If $\e([D])\geq 1$, then $\e([D+K_X])=64$.
\end{corollary}
In fact, the converse of the previous corollary is true.
\begin{proposition}\label{prop: 1-eff and 64-eff}
    Assume $[D] \neq [0;0,0,0]$. Then, $\e([D]) \geq 1$ if and only if $\e([D+K_X])=64$.
\end{proposition}
\begin{proof}
    It suffices to prove the if part. Since $\e([D+K_X])=64$, we have an effective divisor $D' \in \lvert D + K_X \rvert$ whose truncated symmetric coordinates are of the form
    \[
        \begin{array}{ c | r r r }
            \Bigl( d & \bigl(a\ 01) & \bigl(b\ 01) & \bigl(c\ 01) \Bigr)
        \end{array}.
    \]
    As we observed in Subsection~\ref{subsec: effectiveness test}, there are $2\times2\times2 = 8$ candidate subdivisors of $D'$. These candidates occur as exponents of the monomials in the following formal generating function:
    \[
        (x^{C_1} + x^{C_2+C_3})(x^{A_1} + x^{A_2+A_3})(x^{B_1} + x^{B_2+B_3}).
    \]
    In particular, there exist $i,j,k \in \{1,2\}$ such that $C_i+A_j+B_k$ is a subdivisor of $D'$, namely, $D' = (C_i+A_j+B_k)+D''$ for some effective $D''$. On the other hand, we have
    \[
        C_1+A_1+B_1 = K_X + \tau_0,
    \]
    where $\tau_0 = \bigl( 0 \mid (0\ 01)\ (0\ 01)\ (0\ 01) \bigr)$. Since $C_2-C_1$, $A_2-A_1$, and $B_2-B_1$ are $2$-torsions in $\Pic X$, $C_i+A_j+B_k = K_X + \tau$ for some $\tau \in \Tors \Pic X$. This shows that
    \[
        D'' = D' - (C_i+A_j+B_k) = (D+K_X) - (K_X+\tau) = D + \tau,
    \]
    thus $\e([D]) \geq 1$.
\end{proof}
\begin{proposition}\label{prop: 1,64-eff criterion}
    Let $[D]=[d;a,b,c]$ with $d > 0$, $M = \max\{0,a,b,c\}$, and $\ell = \frac 13 (d+a+b+c)$. Then, we have
    \begin{equation}\label{eq: 1-eff}
        \e([D]) \geq 1 \iff M \leq \ell \leq d.
    \end{equation}
    Moreover,
    \begin{equation}\label{eq: 64-eff}
        \e([D]) \geq 64 \iff d \geq 7\ \text{and}\ \max\{3,M+2\} \leq \ell \leq d-3.
    \end{equation}
\end{proposition}
\begin{proof}
    First we prove the $({\Leftarrow})$ direction of \eqref{eq: 1-eff}. Assume $M \leq \ell \leq d$. Proving that $\e([D])\geq 1$ is equivalent to showing that $[D]$ can be written as a nonnegative sum of $\sZ_i$ with $i \in \{0,3\}$. Also, we observe 
    \[
        [\sZ_3] = [0;1,1,1] + [\sZ_0].
    \]
    Hence, $\sZ_3$ differs from $\sZ_0$ in that it increases the values of $a,b,c$ by $1$.
    
    Define $a_0',b_0',c_0'$ by $a_0' := M-a_0$, etc. Let $d':= a_0'+b_0'+c_0'$,
    \[
        [a_0'A_0 + b_0'B_0 + c_0'C_0] = [ d'; -a_0', -b_0', -c_0' ].
    \]
    If we replace single $A_0$ by $A_3$, then it becomes $[d'; \, 1-a_0',\, 1-b_0',\, 1-c_0']$. By iterating such a replacement $\sZ_0 \mapsto \sZ_3$, we obtain $[d' ;\, k'-a_0',\,k'-b_0',\,k'-c_0']$\ ($0 \leq k' \leq d'$) as a nonnegative sum of $\{\sZ_i\}_{i=0,3}$.
    \begin{description}
        \item[Case 1] Suppose $d' \geq M$. By the above observation, we may write $[d';\, M-a_0',\,M-b_0',\,M-c_0']$ as a nonnegative sum of $\{\sZ_i\}_{i=0,3}$. The inequality $M \leq \ell$ implies that
        \[
            d' = 3M - (a+b+c) \leq 3 \ell - (a+b+c) = d,
        \]
        thus,
        \[
            [d;a,b,c] = [d';\, M-a_0',\,M-b_0',\,M-c_0'] + (d-d')[A_3+B_0+C_0].
        \]
        \item[Case 2] Consider the case $d' < M$. In this case, we cannot achieve $[d';\, M-a_0',\,M-b_0',\,M-c_0']$ by previous method since $M$ is too large. The maximum possible is
        \[
            [a_0' A_3 + b_0' B_3 + c_0' C_3 ] = [d';\, d'-a_0',\, d'-b_0',\, d'-c_0'].
        \]
        We have to increase the $a,b,c$-coordinates by $M-d'$, and the $d$-coordinate by $d-d'$. To do this, we use
        \[
            \left\{
                \begin{array}{l@{}l@{}}
                    [A_3+B_0+C_0] &{}= [3;0,0,0] \\{}
                    [A_3+B_3+C_0] &{}= [3;1,1,1] \\{}
                    [A_3+B_3+C_3] &{}= [3;2,2,2]\,.%
                \end{array}%
            \right.
        \]
        Let $\ell_0,\ell_1,\ell_2 \in \ZZ_{\geq 0}$ satisfy
        \[
            \textstyle\ell_0 + \ell_1 + \ell_2 = \frac 13 (d-d'),\quad \ell_1+2\ell_2 = M-d'.
        \]
        This is always possible since $\ell \leq d$ implies $M-d' \leq \frac 23 (d-d')$. Then,
        \[
            \ell_0[3;0,0,0] + \ell_1[3;1,1,1] + \ell_2[3;2,2,2] = [d-d';\, M-d',\,M-d',\,M-d']
        \]
        Adding this to $[a_0'A_3+b_0'B_3+c_0'C_3]$, we get $[d;a,b,c]$.
    \end{description}
    To prove $(\Rightarrow)$ direction, we proceed by induction on $d$. First, it is easy to check the inequalities when $D \in \{ Z_i \}_{i=0,3}$. Assume $[D] = [d;a,b,c]$ satisfies $\e([D]) \geq 1$ and $M \leq \ell \leq d$. Then, for $[D+A_0] = [d+1; a-1,b,c]$, $\ell' := \frac 13((d+1) + (a-1)+b+c) = \ell$ and $M' \in \{ M-1, M\}$, thus the desired inequality $M' \leq \ell' \leq d+1$ holds. For $[D+A_3] = [d+1; a,b+1,c+1]$ we have $\ell' = \ell+1$, $M' \in \{M, M+1\}$, thus $M' \leq \ell' \leq d+1$ also holds.
    
    Next, let us consider \eqref{eq: 64-eff}. By Proposition~\ref{prop: 1-eff and 64-eff}, $[D]=[d;a,b,c]$ has $\sf e$-number $64$ if and only if $[D] \neq [K_X]$ and $\e([D-K_X]) \geq 1$. The condition $[D] \neq [K_X]$ requires $d \geq 7$. Also,
    \[
        [D-K_X] = [ d-6;\, a-1,b-1,c-1],
    \]
    so we put $\ell' := \frac13 ( (d-6) + (a-1)+(b-1)+(c-1) ) = \ell-3$ and $M' = \max\{0, M-1\}$. Then, the condition $\e([D-K_X]) \geq 1$ is equivalent to $ M' \leq \ell' \leq d-6$, i.e. $M'+3 = \max \{ 3, M+2\} \leq \ell \leq d-3$.
\end{proof}
\bigskip
\section{Main algorithm and its verification}\label{sec: Main Algorithm}
\subsection{Reduced forms}
Throughout this section, we describe the routine for computing $h^p(D)$, given $D \in \Pic X$ in (truncated) symmetric coordinates. Before proceeding, we discuss an elementary, but powerful tool for the computation of cohomology.
\begin{proposition}\label{prop: trimming lemma}
    Let $D$ be a divisor on $X$. The following statements hold.
    \begin{enumerate}
        \item If $(D. A_0) < 0$, then $h^0(D) = h^0(D-A_0)$.
        \item Assume $(D\mathbin .A_0)=0$. In the symmetric coordinates of $D$, suppose $A_0[2]$-component is not equal to $00$. Then, $h^0(D) =h^0(D-A_0)$.
    \end{enumerate}
    The curve $A_0$ in {\normalfont(1)} and {\normalfont(2)} can be replaced by any curve in $\{\sZ_i\}_{i=0,3}$.
\end{proposition}
\begin{proof}
    From the short exact sequence
    \[
        0 \to \mathcal O_X(D-A_0) \to \mathcal O_X(D) \to \mathcal O_{A_0}(D) \to 0,
    \]
    we have $h^0(D) \leq h^0(D-A_0) + h^0(\mathcal O_{A_0}(D))$. The conditions in (1) and (2) imply $H^0(\mathcal O_{A_0}(D)) = 0$, hence $h^0(D) = h^0(D-A_0)$.
\end{proof}
Given a divisor $D \in \Pic X$, we may reduce $D$ via the following procedure. 
\begin{enumerate}
    \item If $(D.K_X) >  0$, go to (2). Otherwise, stop the procedure.
    \item Find a curve $\sZ_i$ with $i \in \{0,3\}$ satisfying one of the hypotheses in Proposition~\ref{prop: trimming lemma}. If there is such a curve, then go to (3). Otherwise, stop the procedure.
    \item Replace $D$ by $D-\sZ_i$, and go back to (2).
\end{enumerate}
By Proposition~\ref{prop: trimming lemma}, we see that this procedure does not change $h^0(D)$. If (3) is invoked, then $(D.K_X)$ drops by $1$, thus the procedure terminates in finite steps.

Since $K_X$ is ample, we have $h^0(D) = 0$ if $(D.K_X) < 0$. Thus, if the procedure terminates by (1), then $h^0(D)=0$ unless $D$ is principal.
\begin{proposition}\label{prop: Kleiman criterion}
    Suppose  $D \in \Pic X$ satisfies
    \[
        (D.A_0),\, (D.A_3),\, (D.B_0),\, (D.B_3),\, (D.C_0),\, (D.C_3) \geq 0.
    \]
    Then, $D$ is nef. Moreover, $D$ is ample if all of the above inequalities are strict.
\end{proposition}
\begin{proof}
    Let $E \in \Pic X$ be an effective divisor. By Theorem~\ref{thm: PicX and EffX}\,(2), $E$ is a nonnegative sum of the curves in $\{ \sZ_i\}_{i=0,1,2,3}$. By Remark~\ref{rmk: Some linEquiv}, we have numerical equivalences
    \[
        A_1 \sim_{\sf num} A_2 \sim_{\sf num}A_0+C_3
    \]
    and etc., thus we may pick $E' \in [E]$ such that $E'$ is a nonnegative sum of the curves in $\{\sZ_i\}_{i=0,3}$. By assumption, $(D.E) = (D.E') \geq 0$, thus $D$ is nef. The last ampleness statement follows from Kleiman's criterion; by Theorem~\ref{thm: PicX and EffX}\,(2), the Mori cone $\op{\overline {NE}}(X)$ is polyhedral, hence $\op{\overline {NE}}(X) = \op{NE}(X)$.
\end{proof}

If the procedure terminates by (2), then $D$ is nef by Proposition~\ref{prop: Kleiman criterion}. Moreover, if $(D.\sZ_i)=0$ for some $\sZ \in \{A,B,C\}$ and $i \in \{0,3\}$, then by Proposition~\ref{prop: trimming lemma}\,(2), the $\sZ_i[2]$-component in the symmetric coordinates of $D$ is $00$.

\begin{definition}\label{def: reduced form}
    A nef divisor $D$ with $(K_X.D) >0$ is of the \emph{reduced form} if for each curve $\sZ_i$\ ($i=0,3$) with $(D.\sZ_i)=0$, the $\sZ_i[2]$-component in the symmetric coordinates is $00$.
\end{definition}

Proposition~\ref{prop: trimming lemma} implies that computing $h^0$ of the reduced form is sufficient for our 
purpose. We emphasize that given an effective divisor, finding a reduced form is a routine procedure.

\begin{lemma}\label{lem: h0*h2=0}
    For any divisor $D \in \Pic X$, either $h^0(D)=0$ or $h^2(D) = 0$.
\end{lemma}
\begin{proof}
    Assume that $s \in H^0(D)$ and $s' \in H^0(K_X-D)$ are nonzero sections. Then, the natural map
    \[
        H^0(D) \otimes H^0(K_X-D) \to H^0(K_X)
    \]
    maps $s \otimes s'$ to $ss' \neq 0$, which contradicts $p_g(X)=0$.
\end{proof}

\subsection{Torsions in $\Pic X$} Let $D$ be an effective divisor of the reduced form. To compute $h^p(D)$, we first categorize the numerical classes, and then study $D+\tau$ for each individual $\tau \in \Tors \Pic X$. It is crucial to understand how $\tau$ affects $h^p(D+\tau)$.
\begin{notation}\label{nota: torsion elements}
    Let $\tau \in \Pic X$ be a torsion element. Then, the symmetric coordinates of $\tau$ are of the form
    \[
        \begin{array}{ c | r r r | r r r}
            \Bigl( 0 & \bigl(0\ \text{\textasteriskcentered\textasteriskcentered}) & \bigl(0\ \text{\textasteriskcentered\textasteriskcentered}) & \bigl(0\ \text{\textasteriskcentered\textasteriskcentered}) &
            \bigl(0\ \text{\textasteriskcentered\textasteriskcentered}) & \bigl(0\ \text{\textasteriskcentered\textasteriskcentered}) & \bigl(0\ \text{\textasteriskcentered\textasteriskcentered})
            \Bigr)
        \end{array}.			
    \]
    The torsion elements contain many redundant zeros in their symmetric coordinates. We rule out those zeros and simply write
    \[
        \tau = \bigl( 	\text{\textasteriskcentered\textasteriskcentered}\ \text{\textasteriskcentered\textasteriskcentered}\ \text{\textasteriskcentered\textasteriskcentered}\ \big\vert\ \text{\textasteriskcentered\textasteriskcentered}\ \text{\textasteriskcentered\textasteriskcentered}\ \text{\textasteriskcentered\textasteriskcentered} \bigr),
    \]
    or even simpler, $\tau = ( \text{\textasteriskcentered\textasteriskcentered}\ \text{\textasteriskcentered\textasteriskcentered}\ \text{\textasteriskcentered\textasteriskcentered} )$ in truncated coordinates. We note that when $\tau$ is written in truncated coordinates, it always refers to the first six digits, namely, the $A_0[2]\oplus B_0[2]\oplus C_0[2]$-components. Lemma~3 in \cite{AlexeevOrlov:Burniat} describes how to recover the full coordinates. Suppose $\tau = ( a_1 a_2\ b_1b_2\ c_1c_2 )$ in truncated form, then
    \begin{equation}\label{eq: torsion dependency}
        \tau = \bigl(\ a_1 a_2\ b_1b_2\ c_1c_2\ \big\vert\ (a_1+b_2)a_2\ (b_1+c_2)b_2\ (c_1+a_2)c_2\ \bigr).
    \end{equation}
\end{notation}
Now, we aim to compute $h^0(K_X + \tau)$. If $\tau \neq 0$, then $h^0(K_X + \tau ) \geq 1$. Indeed, we will show that this number is either $1$ or $2$.
\begin{lemma}\label{lem: D along (A0+B3)}
    Let $D \in \Pic X$ be a divisor satisfying $(D.A_0) = (D.B_3) = 1$. Then,
    \[
        h^0(\mathcal O_{A_0+B_3}(D) ) = \left\{
        \begin{array}{ll}
            2 & \text{if } D\big\vert_{A_0[2]} = 00\ \text{and}\ D\big\vert_{B_3[2]} = 10 \\
            1 & \text{otherwise}.
        \end{array}
        \right.
    \]
\end{lemma}
\begin{proof}
    Consider the short exact sequence
    \[
        0 \to \mathcal O_{A_0 + B_3} (D) \to \mathcal O_{A_0}(D) \oplus \mathcal O_{B_3}( D ) \to \mathcal O_{A_0 \cap B_3} \to 0.
    \]
    Since the curve $A_0$ is not rational and $(D . A_0) =1$, we have $h^0(\mathcal O_{A_0}(D))=1$. Similarly, $h^0(\mathcal O_{B_3}(D)) = 1$. Thus, $h^0(\mathcal O_{A_0+B_3}(D)) \leq 2$, and the equality holds if and only if the nonzero sections of $H^0(\mathcal O_{A_0}(D))$ and $H^0(\mathcal O_{B_3}(D))$ vanish at the point $P := A_0 \cap B_3$ simultaneously. Since $\Pic A_0$-component of the symmetric coordinates of $B_3$ is $(1\ 00)$, the nonzero section of $H^0(\mathcal O_{A_0}(D))$ vanishes at $P$ if and only if $D\big\vert_{A_0[2]} = 00$. Similarly, $\Pic B_3$-component of the symmetric coordinates of $A_0$ is $(1\ 10)$, hence the nonzero section of $H^0(\mathcal O_{B_3}(D))$ vanishes at $P$ if and only if $D\big\vert_{B_3[2]} = 10$.
\end{proof}

\begin{proposition}\label{prop: flexible torsions}
    Let $\tau \in \Pic X$ be a nonzero torsion, and let $K_\tau = K_X + \tau$. Then,
    \[
        h^0(K_\tau) = \left\{
            \begin{array}{ll}
                2 & \text{if }\tau \in \{(10\ 00\ 00),\, (00\ 10\ 00),\, (00\ 00\ 10) \}\\
                1 & \text{otherwise.}
            \end{array}
        \right.
    \]
\end{proposition}
\begin{proof}
    Since $h^0(\tau) = 0$ and $\chi(\tau) = 1$, we have $h^2(\tau) = h^0(K_\tau) \geq 1$. If $\sZ \in \{A,B,C\}$ and $i\in \{0,3\}$, then $\sZ_i$ is an elliptic curve and $( K_\tau . \sZ_i) = 1$. From the short exact sequence
    \[
        0 \to \mathcal O_X(K_\tau - \sZ_i ) \to \mathcal O_X(K_\tau) \to \mathcal O_{\sZ_i}(K_\tau) \to 0,
    \]
    we have $h^0(K_\tau) \leq h^0(K_\tau - \sZ_i) + 1$. Thus, if there exists $\sZ_i$ such that $K_\tau - \sZ_i$ is not effective, then $h^0(K_\tau)=1$. We need to analyze when
    \begin{equation}\label{eq: flexibility theorem; assumption A}
        h^0( K_\tau - \sZ_i) \geq 1\ \text{for each $\sZ \in \{A,B,C\}$ and $i \in \{0,3\}$}.
    \end{equation}
    From now on, we assume \eqref{eq: flexibility theorem; assumption A}. Suppose there exists $\sZ_i$\,($i \in \{0,3\}$) such that $h^0( K_\tau - 2\sZ_i) > 0$. By symmetry\,(Remark~\ref{rmk: symmetry}), we may assume $\sZ_i = A_0$. Let $D_1 \in \lvert K_\tau - 2A_0 \rvert$. Then, $(D_1 . B_3) = (D_1 .C_3) = -1 < 0$, thus $\lvert D_1 \rvert$ contains $B_3 + C_3$ in its base locus. In particular, $D_2 := K_\tau - 2A_0 - B_3 - C_3$ is effective. The symmetric coordinates of $D_2$ are of the form
    \[
        \begin{array}{ c | r r r | r r r}
            \Bigl( 2 & \bigl(1\ {**}) & \bigl(0\ {**}) & \bigl(0\ {**}) & (1\ {**}) & (0\ {**}) & (0\ {**}) \Bigr)
        \end{array}.
    \]
    Consulting Table~\ref{table: symCoords}, we see that $D_2$ is linearly equivalent to either
    \[
        C_1,\ C_2,\ B_3+C_0,\ \text{or }C_3+B_0.
    \]
    In the former two cases, we have $K_\tau = D_2 + C_j = 2A_0 + B_3 + C_3 + C_j$, where $j \in \{1,2\}$. Using Proposition~\ref{prop: trimming lemma} repeatedly, we see that $h^0( K_\tau - C_0) = 0$. By \eqref{eq: flexibility theorem; assumption A}, we have $1 \leq h^0(K_\tau) = h^0( K_\tau - C_0 ) + 1$, thus $h^0(K_\tau) = 1$. In the latter two cases, we have either 
    \begin{align*}
        K_\tau &= 2(A_0+B_3) + C_0 + C_3 = K_X + (10\ 00\ 00),\ \text{or} \\
        K_\tau &= 2(A_0+C_3) + B_0 + B_3 = K_X + (00\ 00\ 10).
    \end{align*}
    In both cases, the base locus of $\lvert K_\tau \rvert$ does not contain $A_0$, hence $h^0(K_\tau) = h^0(K_\tau - A_0) + 1$, and $h^0(K_\tau - A_0) =1 $ can be deduced by applying Proposition~\ref{prop: trimming lemma}.

    It remains to understand the case when
    \begin{equation}\label{eq: flexibility theorem; assumption B}
        h^0( K_\tau - 2\sZ_i)=0\ \text{for all $\sZ \in \{A,B,C\}$ and $i \in \{0,3\}$}.
    \end{equation}
    We first deal with the case where $\tau = ({*1}\ {**}\ {**})$. By \eqref{eq: flexibility theorem; assumption A}, $K_\tau - B_3$ is effective and the $\Pic A_0$-component of its symmetric coordinates is $(0\ {*1})$, hence $\lvert K_\tau - B_3 \rvert$ contains $A_0$ in its base locus. Choose $D_3 = \sum_{\sZ}\sum_{i=0}^3 z_i \sZ_i$ in the linear system $ \lvert K_\tau - B_3 - A_0\rvert$. By the assumption \eqref{eq: flexibility theorem; assumption B}, $D_3$ contains neither $A_0$ nor $B_3$. Since the $\Pic A_0$-component of the symmetric coordinates of $D_3$ is $(1\ {*1})$, $D_3$ contains $C_i$ for suitable $i \in \{1,2\}$ and contains neither $C_{2-i}$ nor $C_3$. Similar argument works for $K_\tau - C_3$; there exists a divisor $D_4 \in \lvert K_\tau - C_3 - A_0 \rvert$ such that $D_4$ contains $C_{2-i}$, and contains neither $A_0$, $C_3$, $C_i$, nor $B_3$. Consider the short exact sequence
    \[
        0 \to \mathcal O_X(K_\tau - B_3 - A_0 ) \to \mathcal O_X(K_\tau) \to \mathcal O_{B_3+A_0}(K_\tau) \to 0.
    \]
    By Lemma~\ref{lem: D along (A0+B3)}, $h^0(\mathcal O_{B_3+A_0}(K_\tau)) = 1$. Hence, $h^0(\mathcal O_X(K_\tau)) = h^0(K_\tau - B_3 - A_0) + 1$. Moreover,
    \[
        h^0(K_\tau - B_3 - A_0) = h^0(D_3) = h^0(D_3-C_i).
    \]
    Since $(K_X \mathbin. D_3 - C_i ) = 2$, we have $h^0(D_3 - C_i) =1$. This shows that $h^0(K_\tau) = 2$. In the linear system $\lvert K_\tau \rvert$, we have found two members, namely, $A_0+B_3+D_3$ and $A_0+C_3+D_4$. These two cannot coincide since $D_3$ does not contain $C_3$. It follows that the base locus of $\lvert K_\tau \rvert $ contains $A_0$. Note that to derive this conclusion, we only used \eqref{eq: flexibility theorem; assumption A}, \eqref{eq: flexibility theorem; assumption B}, and $\tau = ({*1}\ {**}\ {**})$. Since $\tau \big\vert_{A_3[2]} = {*1}$, the same argument applies, and we see that the  base locus of $\lvert K_\tau \rvert$ also contains $A_3$. Now, $D_3$ has to contain $C_i$ and $A_3$. Moreover, $(D_3 - C_i - A_3 \mathbin. C_0)  = -1$, hence $D_3$ should contain $C_0$. As $(K_X \mathbin. D_3) = 4$, we have $D_3 = C_i + A_3 + C_0$. Similarly, $D_4 = C_{2-i} + A_3 + B_0$. Comparing symmetric coordinates, we find that
    \[
        A_0 + B_3 + C_i + A_3 + C_0 \neq A_0 + C_3 + C_{2-i} + A_3 + B_0,
    \]
    which is a contradiction.

    Up to symmetry, it remains to address the case $\tau = ({*0}\ {*0}\ {*0})$. By symmetry again, these cases reduce to $\tau = (10\ 10\ 10)$, $\tau = (00\ 10\ 10)$, and $\tau = (00\ 00\ 10)$. The case $\tau = (00\ 00\ 10)$ is already considered above, so we take care of the former two cases. If $\tau = (10\ 10\ 10)$, we have
    \[
        K_\tau = A_0 + C_3 + B_0 + A_3 + C_0 + B_3.
    \]
    If $h^0(K_\tau) > 1$, then without loss of generality, we may assume that $A_0$ is not in the base locus of $\lvert K_\tau \rvert$. Let $D_5 \in \lvert K_\tau \rvert$ be a member which does not contain $A_0$. Then, since $\Pic A_0$-component of the symmetric coordinates of $K_\tau$ is $(1\ 10)$, $D_5$ contains $C_3$, but does not contain $B_3$. Then, however, $\Pic B_3$-component of the symmetric coordinates of $K_\tau$ is $(1\ 10)$, hence $D_5$ must contain $A_0$. This leads to a contradiction, so $h^0(K_\tau) = 1$. If $\tau = (00\ 10\ 10)$, then $K_\tau = A_1 + A_2 + B_0 + B_3$. By the effectiveness test in Subsection~\ref{subsec: effectiveness test}, we deduce that $K_\tau - A_0$ is not effective. Thus, $h^0(K_\tau) = 1$. This completes the proof.
\end{proof}
\begin{definition}\label{def: flexible torsion}
    A torsion element $\tau \in \Pic X$ is said to be \emph{flexible} if $h^0(K_X+\tau)=2$.
\end{definition}
\subsection{Classifying numerical classes}\label{subsec: num class identification} The first step to compute $h^0(D)$ is to categorize the numerical classes into several cases. First of all, we show that nef implies $\e([D])=64$ except for few cases.
\begin{proposition}\label{prop: exotic classes (nef but not 64-eff)}
    Assume that the numerical class $[D] = [d;a,b,c]$ contains at least one effective divisor. If $[D]$ is nef and $d  \geq 7$, then $\e([D])= 64$ except for the following cases:
    \begin{enumerate}
        \item $[2\ell;\,\ell,0,0]$;
        \item $[2\ell+1;\, \ell-1, 0,0]$;
        \item $[2\ell;\, \ell-1, 1, 0]$;
        \item $[2\ell;\, \ell-1, 0, 1]$;
        \item the classes symmetric\,(Remark~\ref{rmk: symmetry}) to (1--4).
    \end{enumerate}
\end{proposition}
\begin{proof}
    Let $\ell := \frac 13(d+a+b+c)$. Then, nefness implies that $(D.A_0)=a$, $(D.A_3)=\ell-b-c$, etc. are nonnegative. By Proposition~\ref{prop: 1,64-eff criterion}, $\e([D])=64$ if and only if
    \[
        \max\{3,M+2\} \leq \ell \leq d-3, \tag{$\star$}
    \]
    where $M = \max \{a,b,c\}$. Suppose the left hand side fails. Thus, either $\ell < M+2$ or $M=0$ and $\ell \leq 2$. But the latter case does not happen; indeed, if $M=0$ then $a=b=c=0$ by nefness, thus $d = 3l-(a+b+c) \leq 6$. Without loss of generality, assume $M=a > 0$ and $\ell \leq M+1$. By nefness, $\ell-a-b \geq 0$, hence $b \leq 1$. Similarly, $c \leq 1$. Then, $\ell-1 \leq a \leq \ell-b$, hence if $b=1$\,(or $c=1$) then $a=\ell-1$. Otherwise, $a \in \{\ell-1,\ell\}$. This leaves four possibilities
    \[
        [2\ell;\, \ell,0,0],\quad [2\ell+1;\, \ell-1, 0, 0],\quad [2\ell;\, \ell-1, 1, 0],\ \text{and}\quad [2\ell;\, \ell-1, 0, 1]
    \]
    We remark that $[2\ell-1;\, \ell-1,1,1]$ is symmetric to $[2\ell-1;\, \ell-2 ,0,0]$ via $(i \leftrightarrow 3-i)$.
    
    We claim that the right hand side of $(\star)$ always holds. Indeed, we have
    \[
        0 \leq (\ell-a-b)+(\ell-b-c)+(\ell-c-a) = 3\ell-2(a+b+c) = d-(a+b+c),
    \]
    thus $2d \geq (a+b+c)+d = 3l$. Since $d \geq 7$, we have $3d \geq 3\ell + 7 > 3(\ell+2)$, thus $d \geq \ell+3$.
\end{proof}
If $\e([D])=64$ and $D-K_X$ is nef and big, then by Kawamata-Viehweg vanishing theorem, we conclude $h^0(D) = \chi(D)$ and $h^1(D)=h^2(D)=0$. However, the assumption is not always true. For instance, if $D$ is nef but not ample, then $(D.\sZ_i)=0$ for some $\sZ \in \{A,B,C\}$ and $i=0,3$, thus $(D-K_X)$ is not nef.
\begin{proposition}\label{prop: exotic classes: ample and 64-eff}
    Assume that $[D]$ is ample and $\e([D])=64$. If $(D-K_X)^2=0$, then up to symmetry,
    \[
        [D] = [2\ell;\, 0, 0,\ell] + [K_X].
    \]
\end{proposition}
\begin{proof}
    Since $(D-K_X \mathbin. \sZ_i) \geq 0$ for any $\sZ \in \{A,B,C\}$ and $i=0,3$, $D-K_X$ is nef. The statement is purely about the numerical classes, we may assume
    \[
        D-K_X = \sum_\sZ \sum_{i\in\{0,3\}} z_i \sZ_i,\quad z_i \in \ZZ_{\geq 0}.
    \]
    Suppose $a_0 > 0$. Then by nefness and $(D-K_X)^2 = 0$, $(D-K_X . A_0)= -a_0 + b_3 + c_3 = 0$. Then, either $b_3 > 0$ or $c_3 > 0$. Assume $b_3 \neq 0$. Then, $(D-K_X . B_3) = -b_3 + c_0 + a_0 = 0$. Thus we have $a_0 = b_3 =: \ell$, and $c_0 = c_3 = 0$. In particular,
    \[
        [D-K_X] = [ \ell  (A_0+B_3) ] = [2\ell;\, 0,0,\ell]. \qedhere
    \]
\end{proof}
Among the classes with $\sf e$-number $64$, the strictly nef ones remain:
\begin{proposition}\label{prop: reduction of Nef 64-eff}
    Let $D \in [d;0,b,c]$ be a nef divisor such that $\e([D])=64$ and $\mathcal O_X(D) \big\vert_{A_0} = \mathcal O_{A_0}$. Assume further that $[D - (A_0+B_3+C_0)] = [d-3;0,b,c]$ has $\sf e$-number $64$. Then,
    \[
        h^0(D) = h^0(D-A_0)+1.
    \]
\end{proposition}
\begin{proof}
    By assumption and Proposition~\ref{prop: 1-eff and 64-eff}, $[D - (A_0+B_3+C_0) - K_X]$ contains at least one effective divisor, so there exists $\tau \in \Tors \Pic X$ such that $D - (A_0 + B_3 + C_0 + K_X + \tau)$ is effective. Moreover, the intersection number of $D - (A_0+B_3 + C_0 + K_X + \tau)$ and $A_0$ is $-1$, hence there exists an effective divisor $D'$ such that
    \[
        D = 2A_0 + B_3 + C_0 + (K_X+\tau) + D'.
    \]
    We claim that $\tau$ can be chosen in such a way that $\tau\big\vert_{A_0[2]} = 00$. Suppose $\tau\big\vert_{A_0[2]} \neq 00$, then $D'$ should contain at least one of $\{C_1,C_2,C_3\}$\,(cf. Subsection~\ref{subsec: effectiveness test}).
    \begin{enumerate}
        \item If $D' \supset C_i$ for $i=1$ or $i=2$, then
        \[
            [D'] \ni (D'-C_i) + (C_0+B_3).
        \]
        Note that $C_0+B_3$ do not contain $C_1,C_2,C_3$.
        \item If $D' \supset C_3$, then, $(D'-C_3 \mathbin . A_0) < 0$, thus $D'$ also contains $A_0$. Thus,
        \[
            [D'] \ni \bigl( D'-(C_3+A_0) \bigr) + A_3+C_0
        \]
        Again, $A_3+C_0$ do not contain $C_1,C_2,C_3$.
    \end{enumerate}
    By applying (1) and (2) iteratively, we can find an effective divisor in $[D']$ containing none of $C_1,C_2,C_3$. Replacing $D'$ if necessary, we may assume that $D'$ contains none of $C_1,C_2,C_3$. Then, $A_0[2]$-component of $D'$ is $00$, as claimed.
    
    Suppose $\lvert D' \rvert$ does not contain $A_0$ in its base locus. Let $G_1 := D-D'$. Since $[G_1-A_0 -K_X] = [A_0+B_3+C_0]$ is nef and big, Kawamata-Viehweg vanishing theorem asserts $h^0(G_1-A_0) = 3$ and $h^1(G_1-A_0)=0$.
    In the short exact sequence
    \[
        0 \to \mathcal O_X(G_1-A_0) \to \mathcal O_X(G_1) \to \mathcal O_{A_0}(G_1) \to 0,
    \]
    we have $\mathcal O_{A_0}(G_1) = \mathcal O_{A_0}(\tau) = \mathcal O_{A_0}$. It follows that $h^1(\mathcal O_X(G_1))=1$, hence $h^0(G_1) = 4 > h^0(G_1-A_0)$. Therefore, the base locus of $\lvert G_1 \rvert$ does not contain $A_0$. Consequently, $\lvert D \rvert = \lvert G_1 + D' \rvert$ does not contain $A_0$ in its base locus.
    
    Now, we assume that the base locus of $\lvert D' \rvert$ contains $A_0$. We may write $D'$ as a nonnegative sum
    \[
        a_0  A_0 + b_3  B_3 + c_1 C_1 + c_2 C_2 + c_3C_3 + \dots
    \]
    The condition $\mathcal O_X(D')\big\vert_{A_0} = \mathcal O_{A_0}$ imposes the relations
    \[
        a_0 = b_3+c_1+c_2+c_3,\quad\text{and}\quad  c_1 \equiv c_2 \equiv c_3\ (\mod 2).
    \]
    But we have a linear equivalence relation (which can be checked directly using symmetric coordinates)
    \[
        C_1+C_2+C_3 + 3A_0 = A_1+A_2+A_3+ 2B_0 + C_0,
    \]
    so we may assume $c_1,c_2,c_3$ are even. Moreover, since $2C_1 = 2C_2 = 2(C_3+B_0)$, we further assume that $c_1=c_2=0$. Suppose $a_0>0$ is even. Then, $a_0 = b_3+c_3$ and $b_3,c_3$ are even, hence $A_0$ is not in the base locus of $\lvert D' \rvert$. This contradicts our assumption. Thus, $a_0$ is odd, and so is $b_3$. Then, 
    \begin{align*}
        D'' :={}& D'-A_0-B_3 \\
        ={}& \frac{b_3-1}{2} \cdot 2(A_0+B_3) + \frac{c_3}{2} \cdot 2(A_0+C_3) + \dots
    \end{align*}
    does not contain $A_0$ in its base locus. Let $G_2 := D - D'' = 3A_0 + 2B_3 + C_0 + K_X + \tau$. In the short exact sequence
    \[
        0 \to \mathcal O_X(G_2-A_0) \to \mathcal O_X(G_2) \to \mathcal O_{A_0} \to 0,
    \]
    $h^1(G_2-A_0)=h^2(G_2-A_0) = 0$ by Kawamata-Viehweg vanishing theorem. Then, $H^0(\mathcal O_X(G_2)) \to H^0(\mathcal O_{A_0})$ is surjective, i.e. $\lvert G_2\rvert$ does not contain $A_0$ in its base locus. Consequently, $D = G_2 + D''$ and the base loci of $G_2$ and $D''$ do not contain $A_0$, hence $D$ does not contain $A_0$ in its base locus. Now, the result follows from the short exact sequence
    \[
        0 \to \mathcal O_X(D-A_0) \to \mathcal O_X(D) \to \mathcal O_{A_0} \to 0. \qedhere
    \]
\end{proof}
Next step is to identify which classes $[D]$ fail to fulfill the hypotheses of Proposition~\ref{prop: reduction of Nef 64-eff}. Before proceed to the statement, we need to introduce a lemma:
\begin{lemma}\label{lem: Nef 64-eff}
    Suppose that $[D] = [d;0,b,c]$ satisfies $d \geq 9$ and $\ell := \frac 13(d+b+c) \geq 4$. Assume $[D]$ is nef and $\e([D])=64$. We have $(D.B_3) \geq 2$, $(D.C_3)\geq 2$, and $(D.B_3) + (D.C_3) \geq 5$. Moreover, $d \geq \ell +5$.
\end{lemma}
\begin{proof}
    Let $M := \max\{b,c\}$. Since $[D]$ has $\sf e$-number $64$,
    \[
        M+2 \leq \ell \leq d-3.
    \]
    Thus, $(D. B_3) = \ell-c \geq 2$ and $(D. C_3) = \ell-b \geq2 $.
    Assume $\ell = b+2 = c+2$. Then, $(D.A_3) \geq 0$ implies $b+c \leq \ell $, so $b,c \leq 2$. Since we assumed $\ell\geq 4$, $b=c=2$ and $\ell=4$. Then $d = 3\ell - b - c = 8 < 9$, a contradiction. This shows that at least one of $(D. B_3)$ or $(D. C_3)$ is larger than $2$.
    
    It remains to show $d - \ell \geq 5$. We have
    \[
        d - \ell = 2\ell -b - c = \ell + (D. A_3) \geq 4.
    \]
    The equality holds if and only if $(D. A_3)=0$ and $\ell=4$, but it is impossible as $d \geq 9$. 
\end{proof}
\begin{proposition}\label{prop: Nef 64-eff, exceptional cases}
    Let $d \geq 9$, and let $[D] = [d;0,b,c]$ be a nef divisor class such that $\e([D])=64$ but $\e([D- (A_0+B_3+C_0)]) \neq 64$. Then, up to symmetry, $[d;0,b,c]$ is one of the following:
    \begin{enumerate}
        \item $[2\ell;\,0,2,\ell-2]$ or $[2\ell;\,0,\ell-2,2]$;
        \item $[2\ell+1;\,0,1,\ell-2]$ or $[2\ell+1;\,0,\ell-2,1]$;
        \item $[2\ell+2;\,0,0,\ell-2]$;
        \item $[9;0,0,0]$;
    \end{enumerate}
\end{proposition}
\begin{proof}
    Let $M := \max\{0,b,c\}$ and $\ell := \frac 13(d+b+c)$. By Proposition~\ref{prop: 1,64-eff criterion},
    \[
        \max\{3,M+2\} \leq \ell \leq d-3
    \]
    holds, but either $d -3 < 7$ or the inequalities
    \[
        \max\{3,M+2\} \leq \ell-1 \leq d-6 	\tag{$\star$}
    \]
    fail. Suppose $\ell \geq 4$. Then the inequality at the right of ($\star$) is true by Lemma~\ref{lem: Nef 64-eff}, thus it fails only if $M \geq \ell-2$. Since $\e([D])=64$, we have $M = \ell-2$. Assume $c = \ell-2$. Then, $\ell-b-c \geq 0$ implies that $b \leq 2$. Similarly, if $b=\ell-2$ then $c\leq2$. By Lemma~\ref{lem: Nef 64-eff}, $[d;0,b,c]$ belongs\,(up to symmetry) to list (1--3). It remains to consider the case $\ell=3$. As $d = 3\ell - b- c \geq 9$, we have $b=c=0$ and $d=9$.
\end{proof}
\begin{proposition}\label{prop: Nef 64-eff, exceptional cases, reduction}
    Let $d \geq 9$ and let $[d;0,b,c]$ be one of the classes listed in Proposition~\ref{prop: Nef 64-eff, exceptional cases}. Suppose $D \in [d;0,b,c]$ is of the reduced form. Then,
    \[
        h^0(D) = h^0(D-A_0) + 1.
    \]
\end{proposition}
\begin{proof}
    First, we treat the simplest case: $[D] = [9;0,0,0]$. The class contains a unique reduced form: $D = A_0+B_0+C_0 + K_X$. It can be easily verified that $D = A_1+A_2+A_3+3B_0+C_0$, so $\lvert D \rvert$ does not contain $A_0$ as a base locus. 
    
    We observe that $[2\ell+1;0,1,\ell-2]$ and $[2\ell+2;0,0,\ell-2]$ can be obtained by adding $[B_0]$ to $[2\ell;0,2,\ell-2]$ once and twice, respectively. Let $[d;0,b,c] = [2\ell,0,2,\ell-2]$. Then, $D \in [d;0,b,c]$ can be written as the following general form
    \[
        D = \ell (A_0+B_3) - 2B_0 + 2C_0 + \tau.
    \]
    We have $\mathcal O_X(D)\big\vert_{A_i} = \mathcal O_{A_i}(\tau)$ for $i=0,3$, so $D$ is of reduced form if and only if $\tau = (00\ {*0}\ {**})$\,(cf. \eqref{eq: torsion dependency}). Suppose we have the result (that the base locus of $D$ does not contain $A_0$) for smaller $\ell$. Then, we have the same result for $\ell+2$ since $\lvert 2(A_0+B_3) \rvert$ does not contain $A_0$ in its base locus. Hence, it suffices to prove for $\ell =4$ and $\ell=5$.
    
    \begin{description}
        \item[Case 1] $[d;0,b,c]=[8;0,2,2]$. The following table exhibits explicit members in $\lvert D \rvert = \lvert 4 (A_0+B_3) - 2B_0 + 2C_0 + \tau\rvert$.
        \[
            \begin{array}{c|c c c|c}
                \tau & D & \qquad & \tau  & D \\ \cline{1-2}\cline{4-5}
                (00\ 00\ 00) & 2A_2 + 2B_1 & & (00\ 10\ 00) & A_1+A_2+2B_1 \\
                (00\ 00\ 01) & A_1+A_2+A_3+B_0+B_1 & & (00\ 10\ 01) & 2A_2 +A_3 + B_0+B_1 \\
                (00\ 00\ 11) & A_1+A_2+A_3+B_0+B_2 & & (00\ 10\ 11) & 2A_2 +A_3 + B_0+B_2 \\
                (00\ 00\ 10) & B_1+B_2+ 2A_2 & & (00\ 10\ 10) & A_1+A_2+B_1+B_2
            \end{array}
        \]
        None of them contains $A_0$, thus $\lvert D \rvert$ does not contain $A_0$ in its base locus.
        \item[Case 2] $[d;0,b,c]=[10;0,2,3]$. In this case, we consider $\lvert D \rvert = \lvert 5 (A_0+B_3) - 2B_0 + 2C_0 + \tau\rvert$.
        \[
            \begin{array}{c|c c c|c}
                \tau & D &\qquad & \tau & D \\ \cline{1-2}\cline{4-5}
                (00\ 00\ 00) & A_1+A_2+A_3+B_0+B_1+B_2 & & (00\ 10\ 00) & 2A_2+A_3+B_0+B_1+B_2\\
                (00\ 00\ 01) & 3B_2 + 2A_2 & & (00\ 10\ 01) & A_1 + A_2 + 3B_2 \\
                (00\ 00\ 11) & B_1+ 2B_2 + 2A_2 & & (00\ 10\ 11) & A_1+A_2+B_1+2B_2 \\
                (00\ 00\ 10) & A_1+A_2+A_3+B_0+2B_1 & & (00\ 10\ 10) &  2A_2+A_3+B_0+2B_1
            \end{array}
        \]
        From the table, we read that $A_0$ is not in the base locus.
    \end{description}
    This completes the proof for $[D] = [2\ell;0,2,\ell-2]$. We further remark that all the chosen members in $\lvert D \rvert$ contain $A_2$. For an effective divisor containing $A_2$, we can add $\tau = (00\ 01\ 00) = A_2 - (C_0+A_3)$ to replace $A_2$ by $(C_0+A_3)$. Hence, the results obtained so far (that $\lvert D \rvert $ does not contain $A_0$ in its base locus) also apply to $\tau$ of the form $(00\ {**}\ {**})$. 
    
    Next, we consider 
    \[
        D = \ell(A_0+B_3) - B_0 + 2C_0 + \tau \in [2\ell+1;\,0,1,\ell-2].
    \]
    In this case, $\mathcal O_{A_0}(D) = \mathcal O_{A_0}(\tau)$ and $(D. A_3) = 1 > 0$. We have $\tau = (00\ {**}\ {**})$; otherwise, $D$ is not reduced. By the previous arguments, $D$ can be written as a sum of $D' \in \lvert \ell(A_0+B_3) - 2B_0 + 2 C_0 + \tau \rvert$ and $B_0$. We have already seen that $\lvert D'\rvert$ does not contain $A_0$ in its base locus, hence neither does $\lvert D \rvert$.
    Same argument applies to $D = \ell(A_0+B_3) + 2C_0 + \tau \in [2\ell+2;\,0,0,\ell-2]$.
    
    It remains to consider the classes $[2\ell+k;\, 0,\ell-2,2-k]$ for $k=0,1,2$. The case $\ell=4$, $k=0$ overlaps with {\sf Case 1} mentioned above. Suppose $(\ell,k)=(5,0)$. We set
    \[
        D = 5 (A_0 + C_3) - 2 C_3 + 2 B_3 + \tau
    \]
    In terms of symmetric coordinates,
    \[
        D = \Bigl( 
        \begin{array}{c | ccc | ccc}
            10 & (0\ 10) & (3\ 00) & (2\ 00) & (0\ 00) & (3\ 10) & (2\ 00)
        \end{array} \Bigr) + \tau.
    \]
    To have the reduced form, we assume $\tau = (10\ {*1}\ {**})$.
    \[
        \begin{array}{c|c c c|c}
            \tau & D &\qquad & \tau & D \\ \cline{1-2}\cline{4-5}
            (10\ 01\ 00) & A_1+2A_2+2B_1 & & (10\ 11\ 00) & 3A_2+2B_1  \\
            (10\ 01\ 01) & 3A_2 +A_3 + B_0 + B_1 & & (10\ 11\ 01) & A_1 + 2A_2 +A_3 + B_0 + B_1 \\
            (10\ 01\ 11) & 3A_2 +A_3 + B_0 + B_2 & & (10\ 11\ 11) & A_1 + 2A_2 +A_3 + B_0 + B_2 \\
            (10\ 01\ 10) & A_1 + 2A_2 + B_1 +B_2 & & (10\ 11\ 10) & 3 A_2 + B_1 +B_2 \\
        \end{array}
    \]
    Again, none of above contains $A_0$, thus $\lvert D \rvert$ do not contain $A_0$ in its base locus. On the other hand, each of the representatives contains $A_2$, thus swapping $A_2 \leftrightarrow A_3+C_0$, we extend the result to $\tau = (10\ {**}\ {**})$. The remaining cases are done by the similar arguments applied to $[2\ell+1;\, 0,\ell-2,1] = [2\ell;\, 0, \ell-2,2] + [C_0]$ and $[2\ell+2;\, 0,\ell-2,0] = [2\ell;\, 0, \ell-2,2] + [2C_0]$
\end{proof}
Combining Proposition~\ref{prop: reduction of Nef 64-eff} and Proposition~\ref{prop: Nef 64-eff, exceptional cases, reduction}, we obtain the following corollary.
\begin{corollary}\label{cor: reduction of nef and 64-eff}
    Let $d \geq 9$. Assume $D \in [d;\,0,b,c]$ is nef and $\e([D])=64$. If $\mathcal O_{A_0}(D) = \mathcal O_{A_0}$, then
    \[
        h^0(D) = h^0(D-A_0)+1.
    \]
\end{corollary}

\begin{remark}
    Later, we will see that the condition $d \geq 9$ in Corollary~\ref{cor: reduction of nef and 64-eff} is optimal. For instance,
    \[
        D:= A_0 + B_0 + K_X + (B_2-B_1) = A_0+B_3 + 2(A_0+B_3+C_0)
    \]
    has the symmetric coordinates
    \[
        \begin{array}{ c | r r r | r r r}
            \Bigl( 8 & \bigl(0\ 00) & \bigl(0\ 00) & \bigl(1\ 10) & (2\ 00) & (2\ 10) & (3\ 00) \Bigr)
        \end{array},
    \]
    thus it is nef, has $\sf e$-number $64$, and is of the reduced form. However, $h^0(D) = h^0(D-A_0)=3$\,(cf. Proposition~\ref{prop: A0+KX [7;011]} and Proposition~\ref{prop: A0+B0+KX [8;001]}), thus $A_0$ is in the base locus of $D$.
\end{remark}
Using Proposition~\ref{prop: 1-eff and 64-eff}, one can easily identify nef classes with $\sf e$-number $64$ and $d=7,8$: up to symmetry, they are $[A_0+K_X] = [7;\, 0,1,1]$, $[A_0+B_0+K_X] = [8;\,0,0,1]$, $[A_0+A_3+K_X]=[8;\,0,2,2]$, and $[A_0+B_3+K_X] = [8;\, 1,1,2]$. Note that the class $[8;\,1,1,2] = [2;\,0,0,1] + [K_X]$ is considered in Proposition~\ref{prop: exotic classes: ample and 64-eff}, hence we do not consider it separately. It remains to identify the nef and effective divisors with $d \leq 6$.
\begin{proposition}\label{prop: small degree classification}
    Let $D \in [d;a,b,c]$ be a nef divisor such that $\e([D]) \geq 1$ and $d \leq 6$. Then, up to symmetry, either $D$ belongs to the list in Proposition~\ref{prop: exotic classes (nef but not 64-eff)}, $D \in [6;\,0,0,0]$, or $D \in [K_X]$.
\end{proposition}
\begin{proof}
    Let $[D] = [d;\,a,b,c]$ and $\ell := \frac 13 (d+a+b+c)$. By the nefness,
    \[
        a,b,c\geq 0 \text{\quad{}and\quad} \ell \geq \max\{a+b,\, b+c,\, c+a\}.
    \]
    Also, by $\e([D]) \geq 1$,
    \[
        M := \max\{a,b,c\} \leq \ell \leq d.
    \]
    If $M = \ell-1$, then $[D]$ belongs to the list in Proposition~\ref{prop: exotic classes (nef but not 64-eff)}. Assume $M \leq \ell-2$. Then,
    \[
        3\ell-d = a+b+c \leq 3M \leq 3(\ell-2)
    \]
    implies $d=6$ and the inequalities are indeed equalities. Hence, $[D] = [6;\, \ell-2,\ell-2,\ell-2]$. By nefness, $2 \leq \ell \leq 4$. The $i \leftrightarrow 3-i$ symmetry swaps $[6;\,0,0,0] \leftrightarrow [6;\,2,2,2]$. If $\ell=3$, $[6;\,1,1,1] = [K_X]$.
\end{proof}

\subsection{Explicit computations for the remaining cases}
According to Subsection~\ref{subsec: num class identification}, we describe a procedure for computing the cohomology dimensions of $D$.
\begin{algorithm}\label{Main Algorithm}\ 
\begin{enumerate}
    \item Test whether either $D$ or $K_X-D$ is effective. If both are not effective, then $h^0(D) = h^2(D)=0$ and $h^1(D) = -\chi(D)$. Otherwise, replacing $D$ by $K_X-D$ if necessary, we assume $D$ is effective. Then, $h^2(D) = 0$ by Lemma~\ref{lem: h0*h2=0}.
    \item\label{flowchart: take reduced form} If necessary, replace $D$ by its reduced form.
    \item If $D$ has $\sf e$-number${}<64$, go to \ref*{flowchart: not 64-eff}. If $\e([D]) =64$, proceed to the following.
        \begin{enumerate}[label=(\arabic{enumi}.\arabic{enumii})]
            \item If $D$ is ample and $(D-K_X)^2>0$, then by Kawamata-Viehweg vanishing theorem, $h^1(D) = 0$, thus $h^0(D) = \chi(D)$.
            \item\label{flowchart: ample 64-eff exotic} If $D$ is ample, but $(D-K_X)^2=0$, then up to symmetry, $D \in [2\ell;\, 0,0,\ell]+[K_X]$\,(Proposition~\ref{prop: exotic classes: ample and 64-eff}). Proposition~\ref{prop: h^0 of [2l;00l]+[KX]} provides $h^0(D)$ explicitly.
            \item\label{flowchart: nef 64-eff, reduction} If $D$ is strictly nef and $d \geq 9$, then by symmetry, we may assume $\mathcal O_{A_0}(D) = \mathcal O_{A_0}$. By Corollary~\ref{cor: reduction of nef and 64-eff}, $h^0(D) = h^0(D-A_0)+1$. We replace $D$ by $D-A_0$ and go back to \ref*{flowchart: take reduced form}.
            \item\label{flowchart: nef 64-eff, with small d} If $D$ is strictly nef and $d \leq 8$. In this case, $D-K_X$ is an effective divisor of degree${}\leq 2$. Up to symmetry, $[D]$ belongs to either $[7;\,0,1,1]$, $[8;\,0,0,1]$, or $[8;\,0,2,2]$. Propositions~\ref{prop: A0+KX [7;011]}, \ref{prop: A0+B0+KX [8;001]}, and \ref{prop: A0+A3+KX [8;022]} computes $h^0(D)$.
        \end{enumerate}
    \item\label{flowchart: not 64-eff} If $\e([D])<64$, then by Proposition~\ref{prop: exotic classes (nef but not 64-eff)} and Proposition~\ref{prop: small degree classification}, $[D]$ belongs to one of the following\,(up to symmetry): $[K_X]$, $[6;\,0,0,0]$, $[2\ell;\,0,0,\ell]$, $[2\ell+1;\, 0,0,\ell-1]$,  $[2\ell;\,0,1,\ell-1]$, or $[2\ell;\,0,\ell-1,1]$. Propositions~ \ref{prop: [6;000]}, \ref{prop: h^0 of [2l;0,0,l]}, \ref{prop: h^0 of [2l+1;0,0,l-1]}, \ref{prop: h^0 of [2l;0,1,l-1]} and \ref{prop: h^0 of [2l;0,l-1,1]} computes $h^0(D)$ for these cases.
\end{enumerate}
To evaluate $h^0$ of the original $D$, we retrace the algorithm and accumulate the increments obtained in \ref*{flowchart: nef 64-eff, reduction}, and then determine $h^1(D)$ via $\chi(D)$. 
\end{algorithm}
\medskip
In this subsection, we present the evaluation for divisors that appeared in \ref*{flowchart: ample 64-eff exotic}, \ref*{flowchart: nef 64-eff, with small d}, and \ref*{flowchart: not 64-eff}.
\begin{proposition}\label{prop: [6;000]}
    If a divisor $D \in [6;\,0,0,0]$ is of the reduced form, then $D = 2(A_0+C_3+B_0)$. In this case, we have $h^0(D) = 3$.
\end{proposition}
\begin{proof}
    In the short exact sequence
    \[
        0 \to \mathcal O_X(D-A_0-B_0) \to \mathcal O_X(D) \to \mathcal O_{A_0} \oplus \mathcal O_{B_0} \to 0,
    \]
    we have the map $H^0(D) \to H^0(\mathcal O_{A_0} ) \oplus H^0(\mathcal O_{B_0})$. If $s \in H^0(D)$ is the section corresponding to $2(A_3+C_0+B_0)$, then $s\big\vert_{A_0} \neq 0$ and $s\big\vert_{B_0} = 0$. On the other hand, the section $s'$ corresponding to $2(A_0+C_0+B_3)$ satisfies $s\big\vert_{A_0} = 0$ and $s\big\vert_{B_0} \neq 0$. The restriction of $s$ and $s'$ to $A_0 + B_0$ generates $H^0(\mathcal O_{A_0}) \oplus H^0(\mathcal O_{B_0})$. This shows that $h^0(D) = h^0(D-A_0-B_0) + 2$. Applying reduction procedure to $D-A_0-B_0$, we find that $h^0(D-A_0-B_0) = h^0(A_0+C_3+B_0)=1$, thus $h^0(D) = 3$.
\end{proof}

\begin{proposition}\label{prop: A0+KX [7;011]}
    Assume $D = A_0 + K_X + \tau \in [7;\,0,1,1]$ is of the reduced from. Then, $\tau = (00\ {**}\ {**})$. Moreover, we have
    \[
        h^0(D) = \left\{
            \begin{array}{l@{\qquad}l}
                3 & \text{if }\tau = (00\ 10\ 00)\\
                1 & \text{if }\tau = (00\ 00\ 00) \\
                2 & \text{otherwise}.
            \end{array}
        \right.
    \]
\end{proposition}
\begin{proof}
    Since $A_0+K_X$ has the symmetric coordinates
    \[
       \begin{array}{ c | r r r | r r r}
            \Bigl( 7 & \bigl(0\ 00) & \bigl(1\ 00) & \bigl(1\ 00) & (1\ 00) & (2\ 10) & (2\ 00) \Bigr)
        \end{array},
    \]
    $D$ is of the reduced form if and only if $\tau = (00\ {**}\ {**})$. Consider the short exact sequence
    \[
        0 \to \mathcal O_X(K_X+\tau) \to \mathcal O_X(A_0+K_X+\tau) \to \mathcal O_{A_0}\to 0.
    \]
    Note that $\mathcal O_{A_0}(A_0+K_X+\tau) = \mathcal O_{A_0}$ by adjunction formula and $\mathcal O_X(\tau)\big\vert_{A_0} = \mathcal O_{A_0}$. If $\tau = (00\ 00\ 00)$, then $h^0(K_X)=h^1(K_X)=0$ implies $h^0(A_0+K_X) = h^0(\mathcal O_{A_0})=1$. Assume $\tau \neq (00\ 00\ 00)$. We have two cases: whether $\tau$ is flexible or not\,(Definition~\ref{def: flexible torsion}). Assume $\tau$ is not flexible. Then, $h^1(K_X+\tau)=0$ implies that $h^0(A_0+K_X+\tau) = h^0(K_X+\tau) + h^0(\mathcal O_{A_0}) = 2$. If $\tau$ is flexible, then either $\tau = B_2-B_1$ or $\tau = A_2-A_1$. If $\tau = B_2-B_1$, then in the short exact sequence
\[
    0 \to \mathcal O_X(D-C_0) \to \mathcal O_X(D) \to \mathcal O_{C_0}(D) \to 0,
\]
    we have $D-C_0 = A_0 + 2A_3 + B_0 + B_3 + C_0$ and its reduced form is the principal divisor. Also, $(D.C_0)=1$ implies $h^0(\mathcal O_{C_0}(D))=1$. Thus, 
    \[
        2 \leq h^0(K_X+\tau) \leq h^0(\mathcal O(D)) \leq h^0(D-C_0) + h^0(\mathcal O_{C_0}(D)) = 2.
    \]
    Finally, if $\tau = A_2-A_1$, then,
    \[
        D = A_3 + 2A_0 + 2B_3 + 2C_0.
    \]
    The linear system $\lvert D \rvert$ contains two pencils given by $\lvert 2(A_0 + B_3) \rvert$ and $\lvert 2(B_3+C_0) \rvert$, so $h^0(D) \geq 3$. Moreover, $h^0(D) \leq h^0(K_X+\tau) + h^0(\mathcal O_{A_0}) = 3$, thus $h^0(D) = 3$.
\end{proof}
\begin{proposition}\label{prop: A0+B0+KX [8;001]}
    Assume $D = A_0+B_0+K_X + \tau \in [8;0,0,1]$ is of the reduced form. Then, $\tau = (00\ 00\ {**})$. Moreover, we have
    \[
        h^0(D) = \left\{
            \begin{array}{l@{\qquad}l}
                 3 & \text{if }\tau \neq (00\ 00\ 00)  \\
                 2 & \text{if }\tau = (00\ 00\ 00).
            \end{array}
        \right.
    \]
\end{proposition}
\begin{proof}
    The symmetric coordinates of $A_0+B_0+K_X$ are
    \[
       \begin{array}{ c | r r r | r r r}
            \Bigl( 8 & \bigl(0\ 00) & \bigl(0\ 00) & \bigl(1\ 00) & (2\ 00) & (2\ 10) & (3\ 00) \Bigr)
        \end{array},
    \]
    hence $D$ is of the reduced form if and only if $\tau = (00\ 00\ {**})$. Since $A_0+C_3+B_0$ is nef and big, by Kawamata-Viehweg vanishing theorem, we have $h^0(D+C_3) = \chi(D+C_3) =3$. Thus, $h^0(D) \leq 3$. Consider the short exact sequence
    \[
        0 \to \mathcal O_X(K_X+\tau) \to \mathcal O_X(D) \to \mathcal O_{A_0} \oplus \mathcal O_{B_0} \to 0.
    \]
    If $\tau=(00\ 00\ 00)$, $h^0(D) = h^0(\mathcal O_{A_0}) + h^0(\mathcal O_{B_0}) = 2$. If $\tau = (00\ 00\ 10) = B_2-B_1$, then $h^1(D) \geq 2$ as $H^1(D) \to H^1(\mathcal O_{A_0} \oplus \mathcal O_{B_0})$ is surjective. It follows that $h^0(D) = 3$. If $\tau \neq (00\ 00\ 00)$ is not flexible, then $h^1(K_X+\tau)=0$, so $h^0(D) = h^0(K_X+\tau) + h^0(\mathcal O_{A_0} \oplus \mathcal O_{B_0}) = 3$.
\end{proof}
\begin{proposition}\label{prop: A0+A3+KX [8;022]}
    Assume $D = A_0 + A_3 + K_X + \tau \in [8;0,2,2]$ is of the reduced form. Then $\tau = (00\ {*0}\ {**})$, and we have
    \[
        h^0(D) = \left\{
            \begin{array}{l@{\qquad}l}
                4 & \text{if }\tau = (00\ 10\ 00)  \\
                2 & \text{if }\tau = (00\ 00\ 00) \\
                3 & \text{otherwise.}
            \end{array}
        \right.
    \]
\end{proposition}
\begin{proof}
    Since $A_0+A_3+K_X$ has the symmetric coordinates
    \[
       \begin{array}{ c | r r r | r r r}
            \Bigl( 8 & \bigl(0\ 00) & \bigl(2\ 10) & \bigl(2\ 00) & (0\ 00) & (2\ 10) & (2\ 00) \Bigr)
        \end{array},
    \]
    $D$ is of the reduced form if and only if $\tau=(00\ {*0}\ {**})$\,(cf. \eqref{eq: torsion dependency}).
    Since $A_0+A_3+C_1$ is nef and big, Kawamata-Viewheg vanishing theorem reads $h^0(D+C_1)=\chi(D+C_1)=4$. Hence, we have an upper bound $h^0(D) \leq 4$. Suppose $\tau = A_2-A_1 = (00\ 10\ 00)$. Then,
    \[
        D = 4A_0+2B_3+2C_3 = 2(A_0+B_3) + 2(C_3+A_0).
    \]
    In right hand side, $h^0(2(A_0+B_3)) = h^0( 2(C_3+A_0)) = 2$, and the corresponding subspaces in $h^0(D)$ are independent, so $h^0(D)=4$. Suppose $\tau = B_2 -B_1 = (00\ 00\ 10)$. In this case,
    \[
        D+C_1 = C_0+ C_2+C_3+A_0+A_3+B_1+B_2.
    \]
    The right hand side does not contain $C_1$, hence the map $h^0(D) \to h^0(D+C_1)$ cannot be surjective, showing that $h^0(D) < h^0(D+C_1)=4$. Furthermore, in the short exact sequence
    \[
        0 \to \mathcal O_X(K_X+\tau) \to \mathcal O_X(D) \to \mathcal O_{A_0} \oplus \mathcal O_{A_3} \to 0,
    \]
    we obtain $h^1(D) \geq 2$. Consequently, $h^0(D) = 3$. Now, assume $\tau$ is not flexible. Then, in the above sequence we easily obtain $h^0(D) = h^0(K_X+\tau) + h^0(\mathcal O_{A_0} \oplus \mathcal O_{A_3}) = 3$. Finally, if $\tau = (00\ 00\ 00)$, then $h^0(D) = h^0(\mathcal O_{A_0} \oplus \mathcal O_{A_3}) = 2$.
\end{proof}

\begin{proposition}\label{prop: h^0 of [2l;0,0,l]}
    Let $D = \ell(A_0+B_3) + \tau \in [2\ell;\,0,0,\ell]$. If $D$ is of the reduced form, then,
    \[
        \tau = \left \{
            \begin{array}{ll}
                (00\ 00\ {*0}) & \text{if $\ell$ is even} \\
                (00\ 00\ {*1}) & \text{if $\ell$ is odd}.
            \end{array}
        \right.
    \]
    For $\ell \geq 1$, we have 
    \[
        h^0(D) = \left\{
            \begin{array}{ll}
                \frac{\ell}{2} + 1 & \text{if $\ell$ is even and }\tau = (00\ 00\ 00) \\
                \frac{\ell}{2} & \text{if $\ell$ is even and }\tau = (00\ 00\ 10) \\
                \frac{\ell+1}{2} & \text{if $\ell$ is odd}
            \end{array}
        \right.
    \]
\end{proposition}
\begin{proof}
    Symmetric coordinates of $\ell(A_0+B_3)$ are
    \[
        \begin{array}{ c | r r r | r r r}
            \Bigl( 2\ell & \bigl(0\ 00) & \bigl(0\ 00) & \bigl(\ell\ \ell0) & (0\ 00) & (0\ \ell0) & (\ell\ 00) \Bigr)
        \end{array}.
    \]
    Using \eqref{eq: torsion dependency}, one can easily see that $D$ is of the reduced form if and only if $\tau$ is as in the statement.

    Suppose $\ell$ is even. Then, either $\tau = 0$ or $\tau = B_2 - B_1$. If $\tau = 0$, we consider the short exact sequence
    \[
        0 \to \mathcal O_X(D-F) \to \mathcal O_X(D) \to \mathcal O_F(D) \to 0,
    \]
    where $F \in \lvert 2(A_0+B_3) \rvert $ is a general member. Then, $F$ does not intersect with $A_0 + B_3$, hence we have $\mathcal O_F(D) = \mathcal O_F$. Furthermore, $\lvert D \rvert$ does not contain $F$ as a base locus, so the map $H^0(\mathcal O_X(D)) \to H^0(\mathcal O_F)$ is surjective. This shows that $h^0(D) = h^0(D-F) + 1$. Since $D-F = (\ell-2)(A_0+B_3)$, we have
    \[
        h^0( D ) = h^0( 2(A_0+B_3) ) + \frac {\ell-2} 2 = \frac \ell 2 + 1.
    \]
    If $\tau = B_2 - B_1$, then $D = (\ell-2) (A_0+B_3) + B_1 + B_2$. Since $D + B_1 + B_2 = (\ell+2)(A_0+B_3)$, we have the short exact sequence
    \[
        0 \to \mathcal O_X(D) \to \mathcal O_X((\ell+2)(A_0+B_3)) \to \mathcal O_{B_1} \oplus \mathcal O_{B_2} \to 0.
    \]
    Since $\lvert (\ell+2)(A_0+B_3) \rvert$ contains $(\frac \ell 2 +1) B_1$ and $ (\frac \ell 2 + 1)B_2$, the map $H^0(\mathcal O_X( (\ell+2) (A_0+B_3) )) \to H^0( \mathcal O_{B_1} ) \oplus H^0(\mathcal O_{B_2})$ is surjective. This shows that $h^0(D) = h^0((\ell+2) (A_0+B_3)) - 2 = \frac \ell 2$.

    Now, suppose $\ell$ is odd. Either $\tau = B_2 - (A_0+B_3)$ or $\tau = B_1 - (A_0+B_3)$, thus $D = (\ell-1)(A_0+B_3) + B_i$ for $i \in \{1,2\}$. Since $D+ B_i = (\ell+1)(A_0+B_3)$, we have the short exact sequence
    \[
        0 \to \mathcal O_X(D) \to \mathcal O_X((\ell+1)(A_0+B_3)) \to \mathcal O_{B_i} \to 0,
    \]
    from which we read $h^0(D) = h^0( (\ell+1) (A_0+B_3)) - 1 = \frac{\ell+1}{2}$.
\end{proof}

\begin{proposition}\label{prop: h^0 of [2l+1;0,0,l-1]}
    If $D = \ell(A_0+B_3) + C_0 + \tau \in [2\ell+1;\, 0,0,\ell-1]$ is of the reduced form, then $\tau = (00\ 00\ {**})$. For $\ell \geq 1$, we have
    \begin{equation}\label{eq: [2l+1;0,0,l-1]}
        h^0(D) = \left\{
            \begin{array}{l@{\qquad}l}
                \lfloor \frac \ell 2 \rfloor + 1 & \text{if }\tau = (00\ 00\ 00) \\
                 \lfloor \frac{\ell-1}{2} \rfloor + 1 & \text{if }\tau = (00\ 00\ {*1}) \\
                \lfloor \frac{\ell}{2} \rfloor & \text{if }\tau = (00\ 00\ 10).
            \end{array}
        \right.
    \end{equation}
\end{proposition}
\begin{proof}
    The symmetric coordinates of $D-\tau$ are
    \[
        \begin{array}{ c | r r r | r r r}
            \Bigl( 2\ell+1 & \bigl(0\ 00) & \bigl(0\ 00) & \bigl(\ell-1\ \ell0) & (1\ 10) & (1\ \ell0) & (\ell\ 00) \Bigr)
        \end{array}.
    \]
    From the description it is clear that $D$ is of the reduced form only if $\tau = (00\ 00\ {**})$ except when $\ell=1$. If $\ell=1$, $D$ is of the reduced form if and only if $\tau = (00\ 00\ 10)$.

    Let $F \in \lvert 2(A_0+B_3)\rvert$ be a general member. In the short exact sequence
    \[
        0 \to \mathcal O(C_0-A_0) \to \mathcal O(A_0+2B_3+C_0) \to \mathcal O_F(C_0) \to 0,
    \]
    we have $h^2(C_0-A_0) = h^0(K_X+A_0-C_0)= 0$; indeed, $K_X+A_0 = A_1+A_2+A_3+2B_0$ and $h^0(K_X+A_0)=1$\,(cf. Proposition~\ref{prop: A0+KX [7;011]}). Since $\chi(C_0-A_0)=0$, $h^p(C_0-A_0)=0$ for all $p$. Moreover, $A_0+2B_3+C_0$ has the principal divisor as the reduced form, so $h^0(\mathcal O_F(C_0)) = h^0(A_0+2B_3+C_0)=1$. 
    
    Since $\tau = (00\ 00\ {**}) \in \{ 0,\, B_2-(B_3+A_0),\, B_1-(B_3+A_0),\, B_2-B_1\}$, $\mathcal O_F(\tau) = \mathcal O_F$. In the short exact sequence,
    \begin{equation}\label{eq: [2l+1;0,0,l-1], SES from general fiber}
        0 \to \mathcal O_X(D-F) \to \mathcal O_X(D) \to \mathcal O_F(C_0) \to 0,
    \end{equation}
    we have $h^0(D) \leq h^0(D-F) + 1$. Since the members of $\lvert F\rvert$ sweep out the whole surface $X$, the map $H^0(D) \to H^0(\mathcal O_F(C_0))$ is nonzero if $\lvert D \rvert$ is nonempty. In particular, $h^0(D) = h^0(D-F)+1$ if $h^0(D) > 0$. The truncated symmetric coordinates of $F$ are 
    \[
    \begin{array}{ c | r r r }
        \Bigl( 4 & \bigl(0\ 00) & \bigl(0\ 00) & \bigl(2\ 00)\Bigr)
    \end{array},
    \]
    hence if $h^0(D-F)$ satisfies \eqref{eq: [2l+1;0,0,l-1]} for $\ell$ substituted by $\ell-2$, then \eqref{eq: [2l+1;0,0,l-1]} holds for $h^0(D)$, with one potential exception when $\ell=3$ and $\tau = (00\ 00\ 10)$. In this exceptional case, we need to know $h^0(D) > 0$. It can be directly seen by $D = A_0 + B_3 + C_0 + B_1 + B_2$ in such a case.

    It remains to prove the statement for $\ell=1$ and $\ell=2$. Assume $\ell=1$. Since $(00\ 00\ 11) = B_1 - (A_0+B_3)$, $(00\ 00\ 01) = B_2 - (A_0+ B_3)$, and $(00\ 00\ 10) = B_2 - B_1$, we have
    \[
        D = \left \{
        \begin{array}{ll}
            A_0+B_3+C_0 & \text{if } \tau = (00\ 00\ 00) \\
            C_0 + B_1 & \text{if }\tau = (00\ 00\ 11) \\
            C_0 + B_2 & \text{if }\tau = (00\ 00\ 01) \\
            A_0+B_3+C_0+B_2-B_1 &\text{if }\tau = (00\ 00\ 10) \\
        \end{array}
        \right.
    \]
    It can be easily checked that $h^0(A_0+B_3+C_0) = h^0(C_0+ B_1) = h^0(C_0+B_2) = 1$ and $h^0(A_0+B_3+C_0+B_2 - B_1)=0$, verifying \eqref{eq: [2l+1;0,0,l-1]} for $\ell=1$. For $\ell=2$, we invoke the short exact sequence \eqref{eq: [2l+1;0,0,l-1], SES from general fiber}. If $\tau \neq (00\ 00\ 00)$, then $D-F = C_0+\tau$ is not effective. Hence, $h^0(D) = 1$. Moreover, $h^0(C_0+F) = 2$, namely, $h^0(D)=2$ for $\tau = (00\ 00\ 00)$. This verifies \eqref{eq: [2l+1;0,0,l-1]} for $\ell=2$.
\end{proof}

\begin{proposition}\label{prop: h^0 of [2l;0,1,l-1]}
    If $D = (\ell-1)(A_0+B_3) + (C_0+A_3) + \tau \in [2\ell;\, 0,1,\ell-1]$ is of the reduced form, then
    \[
        \tau = \left \{
        \begin{array}{ll}
            (00\ {*1}\ {**}) & \text{if }\ell \geq 2 \\
            (00\ {*1}\ {00}) & \text{if }\ell = 1.
        \end{array}
        \right.
    \]
    Moreover, we have $h^0(D) = \ell - 1$ if $\ell \geq 2$, and $h^0(D) = 1$ if $\ell=1$.
\end{proposition}
\begin{proof}
    The symmetric coordinates of $D-\tau$ are
    \[
        \begin{array}{ c | r r r | r r r}
            \Bigl( 2\ell & \bigl(0\ 00) & \bigl(1\ 10) & \bigl(\ell-1\ (\ell-1)0) & (0\ 10) & (1\ (\ell-1)0) & (\ell-1\ 00) \Bigr)
        \end{array}.
    \]
    Using \eqref{eq: torsion dependency}, one can see that $D$ is of the reduced form if and only if
    \[
        \tau = (00\ {*}1\ {**}\ \vert\ 01\ {**}\ {**})
    \]
    when $\ell \geq 2$, and $\tau = (00\ {*1}\ 00)$ when $\ell=1$.

    If $\ell=1$, then either $\tau = A_1 - (C_0+A_3)$ or $\tau = A_2 - (C_0+A_3)$, thus $D = A_i$ for $i \in \{1,2\}$. To compute $h^0(A_i)$, consider the short exact sequence
    \[
        0 \to \mathcal O_X(A_i) \to \mathcal O_X(2A_i) \to \mathcal O_{A_i} \to 0.
    \]
    Recall that $\lvert 2A_i \rvert$ contains $2(C_0+A_3)$, hence $\mathcal O_{A_i}(2A_i) = \mathcal O_{A_i}$ and the map $H^0(\mathcal O_X(2A_i)) \to H^0(\mathcal O_{A_i})$ is surjective. By Proposition~\ref{prop: h^0 of [2l;0,0,l]}, $h^0(2A_i) = h^0(2(C_0+A_3))= 2$, hence $h^0(A_i) = h^0(2A_i) - 1 = 1$.

    Assume $\ell \geq 2$. Since $(00\ {*1}\ 00) \in \{ A_1 - (C_0+A_3), \, A_2 - (C_0+A_3)\}$ and $(00\ 00\ {**}) \in \{ 0,\, B_2-B_1,\, B_1 - (A_0+B_3), \, B_2 - (A_0+B_3) \}$, we have that
    \[
        D = \left \{
            \begin{array}{ll}
                (\ell-1)(A_0+B_3) + A_i & \text{if }\tau = (00\ {*1}\ 00)  \\
                (\ell-2)(A_0+B_3) + A_i + B_j & \text{if } \tau = (00\ {*1}\ {*1} ) \\
                (\ell-1)(A_0+B_3) + A_i + B_1 - B_2 & \text{if } \tau = (00\ {*1}\ {10} )
            \end{array}
        \right.
    \]
    for suitable $i,j \in \{1,2\}$.
    
    Let $F \in \lvert 2(A_0+B_3) \rvert$ be a general member. Since $F$ is disjoint from $A_0, A_3, B_0, B_1, B_2, B_3$, we have $\mathcal O_F(D) = \mathcal O_F(A_i)$ or $\mathcal O_F(D) =\mathcal O_F(A_{2-i})$. In particular, $h^0(\mathcal O_F(D)) \leq \deg \mathcal O_F(D) = 2$. We claim that if $\ell \geq 4$, then the image of the map $H^0(\mathcal O_X(D)) \to H^0(\mathcal O_F(D))$ is two dimensional, hence surjective. Using Remark~\ref{rmk: Some linEquiv}, we observe that
    \begin{align*}
        3(A_0+B_3) + A_i &= (A_0+B_3) + A_i + 2 B_{2-i} \\
        &= (A_3 + B_0) + A_{2-i} + B_1 + B_2,
    \end{align*}
    and that
    \begin{align*}
        (A_0+B_3) + A_i + B_j &= (A_3+B_0)+A_{2-i}+B_{2-j}.
    \end{align*}
    The linear system $\lvert (\ell-1)(A_0+B_3) + A_i\rvert $ contains $ (\ell-1)(A_0+B_3) + A_i$ and $(\ell-4)(A_0+B_3) + (A_3+B_0) + A_{2-i} + B_1 + B_2$, whose restrictions to $F$ are $A_1 \big\vert_F$ and $A_{1} \big\vert_F$. It follows that the image of $H^0(\mathcal O_X(D)) \to H^0(\mathcal O_F(D))$ contains two sections whose associated divisors are $A_1\big\vert_F$ and $A_2 \big\vert_F$. This shows that the image of $H^0(\mathcal O_X(D)) \to H^0(\mathcal O_F(D))$ is two dimensional if $\tau = (00\ {*1}\ 00)$. The other cases can be proved similarly. Indeed, we can see that the linear system $\lvert (\ell-2)(A_0+B_3) + A_i + B_j \rvert$ contains
    \begin{equation}\label{eq: tmp: effective representatives 1}
         (\ell-2)(A_0+B_3) + A_i + B_j \quad \text{and}\quad (\ell-3) + (A_3+B_0) + A_{2-i} + B_{2-j},
    \end{equation}
    while $\lvert  (\ell-1)(A_0+B_3) + A_i + B_1 - B_2\rvert $ contains
    \begin{equation}\label{eq: tmp: effective representatives 2}
        (\ell-3) (A_0+B_3) + A_i + B_1 + B_2 \quad \text{and}\quad (\ell-2)(A_0+B_3) + (B_0+A_3) + A_{2-i}.
    \end{equation}
    These presentations together with the previous argument assert that $H^0(\mathcal O_X(D)) \to H^0(\mathcal O_F(D))$ is surjective. So far we proved that if $\ell \geq 4$ then $h^0(D) = h^0(D-F) + 2$. Thus, if we show $h^0(D)= \ell-1$ for $\ell=2$ and $\ell=3$, then the computation for higher $\ell$ follows by induction.

    Assume $\ell=2$. Depending on $\tau$, we have $D = A_0+B_3 + A_i$, $D=A_i+B_j$, or $D=B_0+A_3 +A_i$. We will show that $h^0(D) = 1$ for all the three cases. Indeed, we have $h^0(D-A_3)=0$ and $\mathcal O_X(D) \big\vert_{A_3} = \mathcal O_{A_3}$, hence from the short exact sequence
    \[
        0 \to \mathcal O_X(D-A_3) \to \mathcal O_X(D) \to \mathcal O_{A_3} \to 0
    \]
    we have $h^0(D) = 1 = \ell-1$ for $\ell=2$.

    Now, consider $\ell=3$. We note that the above surjectivity argument for $H^0(\mathcal O_X(D)) \to H^0(\mathcal O_F(D))$ still applies when $\tau \neq (00\ {*1}\ 00)$, as the divisors in \eqref{eq: tmp: effective representatives 1} and \eqref{eq: tmp: effective representatives 2} are members of the corresponding linear systems. Thus, if $\tau \neq (00\ {*1}\ 00)$, $h^0(D) = h^0(D-F) + 2$. One can directly show that neither $D-F = A_i + B_j - (A_0+B_3)$ nor $D-F = A_i + B_1 - B_2$ is effective, showing that $h^0(D) = 2$ for $\tau \neq (00\ {*1}\ 00)$. Now, assume $\tau = (00\ {*1}\ 00)$, i.e. $D = 2(A_0+B_3) + A_i$. From the short exact sequence
    \[
        0 \to \mathcal O_X(A_0 + 2B_3 + A_i) \to \mathcal O_X(D) \to \mathcal O_{A_0} \to 0,
    \]
    we have $h^0(D) = h^0(A_0+2B_3+A_i)+1$. Moreover, $h^0(A_0+2B_3 + A_i) = h^0(A_0+B_3 + A_i)$ by $(A_0 + 2B_3 + A_i \mathbin . B_3 ) < 0$, and it equals to $1$ as computed in the $\ell=2$ case above. This shows that $h^0(D) = 2 = \ell-1$ for $\ell = 3$.
\end{proof}

\begin{proposition}\label{prop: h^0 of [2l;0,l-1,1]}
    For $\ell = 1,2$, $[2\ell;\,0,\ell-1,1]$ and $[2\ell;\,0,1,\ell-1]$ are symmetric to each other. Assume $\ell \geq 3$. If $D = (\ell-1) (C_0+A_3) + (A_0+ B_3) + \tau \in [2\ell;\, 0,\ell-1,1]$ is of the reduced form, then
    \[
        \tau = \left \{
            \begin{array}{ll}
                (00\ {*1}\ {**}) & \text{if $\ell$ is even} \\
                (00\ {*0}\ {**}) & \text{if $\ell$ odd}.
            \end{array}
        \right.
    \]
    Moreover, we have
    \[
        h^0(D) = \left\{
            \begin{array}{ll}
                \frac{\ell}{2} & \text{if $\ell$ is even}\\
                \frac{\ell+1}{2} & \text{if $\ell$ is odd and }\tau \in \{ (00\ 00\ 00),\, (00\ 00\ 01), \, (00\ 00\ 11),\, (00\ 10\ 10) \} \\
                \frac{\ell-1}{2} & \text{if $\ell$ is odd and }\tau \in \{ (00\ 10\ 00),\, (00\ 10\ 01), \, (00\ 10\ 11),\, (00\ 00\ 10) \}.
            \end{array}
        \right.
    \]
\end{proposition}
\begin{proof}
    The symmetric coordinates of $D - \tau$ are
    \[
        \begin{array}{ c | r r r | r r r}
            \Bigl( 2\ell & \bigl(0\ 00) & \bigl(\ell-1\ (\ell-1)0) & \bigl(1\ 10) & (0\ (\ell-1)0) & (\ell-1\ 10) & (1\ 00) \Bigr)
        \end{array},
    \]
    thus for $\ell \geq 2$, $D$ is of the reduced form if and only if $\tau = (00\ {*1}\ {**})$ when $\ell$ is even, and $\tau = (00\ {*0}\ {**})$ when $\ell$ is odd. We claim that if $\ell \geq 3$, then 
    \begin{equation}\label{eq: reduction for [2l;0,l-1,1]}
        h^0(D) = h^0(D - 2(C_0+A_3)) + 1.
    \end{equation}
    Assume $\ell = 3$. Table~\ref{table: reduced forms in [6;021]} presents explicit members in the linear system $\lvert D \rvert$.
    \[
        \begin{array}{c|c c c|c}
            \tau & D & \qquad & \tau  & D \\ \cline{1-2}\cline{4-5}
            (00\ 00\ 00) & (A_0+B_3) + 2A_1 & & (00\ 10\ 00) &  (A_0+B_3) + A_1 + A_2 \\
            (00\ 00\ 01) & B_2 + 2A_1 & & (00\ 10\ 01) & B_2 + A_1 + A_2 \\
            (00\ 00\ 11) & B_1 + 2A_1 & & (00\ 10\ 11) & B_1 + A_1 + A_2 \\
            (00\ 00\ 10) & A_1 + A_2 + (A_3+B_0) & & (00\ 10\ 10) & 2A_1 + (A_3+B_0)
        \end{array}
    \]\vspace{-0.5\baselineskip}\captionof{table}{Reduced forms in $[6;\, 0,2,1]$}\label{table: reduced forms in [6;021]}\smallskip    
    If $\tau \neq (00\ {*0}\ 10)$, then the members in Table~\ref{table: reduced forms in [6;021]} do not contain $A_3$. Hence, $h^0(D) = h^0(D-A_3) + 1$. Moreover, $D-A_3$ is not of the reduced form as $(D-A_3 \mathbin. C_0) = 0$ and $C_0[2]$-component of $D-A_3$ is not equal to $00$. Thus, $h^0(D - A_3) = h^0(D-A_3-C_0)$. Still, $D-A_3-C_0$ is not of the reduced form, and continuing the procedure, we arrive $h^0(D-A_3-C_0) = h^0(D - 2A_3 - C_0) = h^0(D- 2A_3 - 2C_0)$. This proves \eqref{eq: reduction for [2l;0,l-1,1]} for $\tau \neq (00\ {*0}\ 10)$. Suppose $\tau = (00\ {*0}\ 10)$. Since the base locus of $\lvert D \rvert$ does not contain $A_0$, we have $h^0(D) = h^0(D-A_0) + 1$. Similar to the previous argument, finding the reduced form of $D-A_0$ leads to
    \[
        h^0(D-A_0) = h^0(D-A_0-C_3) = h^0(D - 2A_0 - C_3) = h^0(D - 2( A_0+C_3) ) ,
    \]
    from which \eqref{eq: reduction for [2l;0,l-1,1]} follows. This establishes \eqref{eq: reduction for [2l;0,l-1,1]} for $\ell=3$. 

    For $\ell = 4$, we have $\tau = (00\ {*0}\ {**}) + (A_1 - (C_0+A_3))$, hence the table of reduced forms in $[D]$ can be obtained by adding $A_1$ to each entry of Table~\ref{table: reduced forms in [6;021]}. More generally for $\ell \geq 4$, the table of reduced forms in $[D]$ is obtained by adding $(\ell-3) A_1$ to each entry of Table~\ref{table: reduced forms in [6;021]}. Then, the argument used to prove \eqref{eq: reduction for [2l;0,l-1,1]} for $\ell=3$ applies in precisely the same way for $\ell \geq 4$.

    For $\ell = 2$, $h^0(D) = 1$ by Proposition~\ref{prop: h^0 of [2l;0,1,l-1]}. Thus, for $\ell \geq 4$ even, $h^0(D) = 1 + \frac{\ell-2}{2} = \frac \ell 2$. If $\ell = 1$, then $[D]$ contains precisely $4$ effective divisors, namely, $(A_3+B_0),\, B_1,\, B_2$, and $(A_0+B_3)$. These correspond to $\tau = (00\ 10\ 10)$, $\tau =  (00\ 00\ 11)$, $\tau = (00\ 00\ 01)$, and $ \tau = (00\ 00\ 00)$, respectively. Thus, for $\ell \geq 3$ odd, we have
    \[
        h^0(D) = \left\{
            \begin{array}{ll}
                \frac{\ell+1}{2} & \text{if }\tau \in \{ (00\ 00\ 00),\, (00\ 00\ 01), \, (00\ 00\ 11),\, (00\ 10\ 10) \} \\
                \frac{\ell-1}{2} & \text{if }\tau \in \{ (00\ 10\ 00),\, (00\ 10\ 01), \, (00\ 10\ 11),\, (00\ 00\ 10) \}.
            \end{array}
        \right.
    \]
    This completes the proof.
\end{proof}

\begin{proposition}\label{prop: h^0 of [2l;00l]+[KX]}
    Let $K_\tau = K_X + \tau $ and let $D = \ell( A_0 + B_3) + K_\tau \in [2\ell;\,0,0,\ell] + [K_X]$. For $\ell \geq 1$, we have
    \[
        h^0(D) = \left \{
            \begin{array}{ll}
                \lfloor \frac{3\ell+4}{2} \rfloor & \text{if }\tau = (00\ 10\ 00)  \\
                \lfloor \frac{3\ell+3}{2} \rfloor & \text{if }\tau = (00\ 10\ {*1}) \text{ or }\tau = (00\ 00\ 10) \\
                \lfloor \frac{3\ell+2}{2} \rfloor & \text{if } \tau \neq (00\ 10\ 00)\ \text{and}\ \tau=(00\ \epsilon0\ {* \bar\epsilon}),\ \text{where }\bar\epsilon = 1-\epsilon \\
                \lfloor \frac{3\ell+1}{2} \rfloor & \text{if }\tau = (00\ 00\ 00) \\
                \ell+1 & \text{otherwise}.
            \end{array}
        \right.
    \]
\end{proposition}
\begin{proof}
    We use the short exact sequence
    \begin{equation}\label{eq: [2l;00l]+[KX], SES for reduction}
        0 \to \mathcal O_X(D-(A_0+B_3)) \to \mathcal O_X(D) \to \mathcal O_{A_0+B_3}(D) \to 0
    \end{equation}
    to reduce $\ell$ to smaller values and compute $h^0(D)$ inductively. To do this, we claim that if $\ell \geq 2$, then the map $H^0(\mathcal O_X(D)) \to H^0(\mathcal O_{A_0+B_3}(D))$ is surjective. Since
    \[
        \mathcal O_{A_0+B_3}(D) = \left\{
            \begin{array}{ll}
                \mathcal O_{A_0+B_3}(K_\tau) & \text{if $\ell$ is even} \\
                \mathcal O_{A_0+B_3}(A_0+B_3+K_\tau) & \text{if $\ell$ is odd},
            \end{array}
        \right.
    \]
    we need to divide into two case based on the parity of $\ell$. 
    
\begin{description}
    \item[Step 1] If $\ell$ is even.\par
    Once we prove the surjectivity for $\ell=2$, then the result extends to the higher $\ell$ as 
    \[
        \mathcal O_X(D) = \mathcal O_X(2(A_0+B_3) + K_\tau) \otimes \mathcal O_X( (\ell-2)(A_3+B_0)),
    \]
    and $\mathcal O_X((\ell-2) (A_3+B_0)) \big\vert_{A_0+B_3} = \mathcal O_X(A_0+B_3)$.
    Indeed, if $\bar s \in H^0(\mathcal O_{A_0+B_3}(D))$ and $s_2 \in H^0(\mathcal O_X(2(A_0+B_3) + K_\tau)$ satisfy $s_2\big\vert_{A_0+B_3} = \bar s$, then by taking $s' \in H^0((\ell-2)\mathcal O_X(A_3+B_0))$ which is nonzero along $A_0+B_3$, we see that $s_2 s' \in H^0(D)$ restricts to $\bar s$ up to multiplication by constants.
    
    In the short exact sequence
    \[
        0 \to \mathcal O_X(K_\tau) \to \mathcal O_X( (A_0+B_3) + K_\tau) \to \mathcal O_{A_0+B_3}( A_0+B_3+K_\tau) \to 0,
    \]
    we have $h^0( \mathcal O_{A_0 +B_3} ( A_0 + B_3 + K_\tau)) = 1 $ if and only if $\tau \neq (00\ {\epsilon 0}\ {*\epsilon})$\,(cf. Lemma~\ref{lem: D along (A0+B3)}). By Proposition~\ref{prop: flexible torsions}, we see that $h^1(A_0+B_3+K_\tau) = 0$ if $\tau$ is none of $(10\ 00\ 00)$, $(00\ 10\ 00)$, or $(00\ \epsilon0\ {*\epsilon})$. In such cases, we get the desired surjectivity from \eqref{eq: [2l;00l]+[KX], SES for reduction}. For $\tau = (10\ 00\ 00)$, we have $h^0(\mathcal O_{A_0+B_3}(K_\tau)) = 1$. Moreover, we have
    \[
        4(A_3 + B_0) + C_0 + C_3 \in \lvert D \rvert,
    \]
    so there exists a global section of $\mathcal O_X(D)$ which does not vanish along $A_0+B_3$. In particular, the map $H^0(\mathcal O_X(D)) \to H^0(\mathcal O_{A_0+B_3}(K_\tau))$ is nonzero, namely, surjective. If $\tau = (00\ 10\ 00)$, we have $h^0(\mathcal O_{A_0+B_3}(K_\tau)) = 2$. On the other hand, we can find two members of $\lvert D \rvert$ as follows.
    \[
        A_1 + A_2 + B_0 + B_1 + B_2 + B_3,\quad 2(A_3+B_0) + 2(C_3+B_0) + A_0 + A_3.
    \]
    The corresponding global sections of $\mathcal O_X(D)$ restrict to linearly independent sections along $A_0+B_3$ since the zero loci of these sections are precisely $B_3$ and $A_0$, respectively. This shows that $H^0(\mathcal O_X(D)) \to H^0(\mathcal O_{A_0+B_3}(K_\tau))$ is surjective. For $\tau = (00\ \epsilon 0\ {*\epsilon})$, we have $h^0(\mathcal O_{A_0+B_3}(K_\tau)) = 1$. Thus, it suffices to find a member of $\lvert D \rvert$ which does not contain $A_0+B_3$.
    \[
        \begin{array}{c|c}
            \tau & \text{a member of $\lvert D \rvert$} \\ \hline
            (00\ 00\ 00) & 2(A_3+C_0)+B_0+B_1+B_2+B_3 \\
            (00\ 10\ 01) & 2(A_3+B_0+C_0)+A_3+B_2+B_3 \\
            (00\ 00\ 10) & 2(A_3+C_0)+B_0+ 2 B_1 + B_3 \\
            (00\ 10\ 11) & 2(A_3+B_0+C_0)+A_3+B_1+B_3
        \end{array}
    \]
    This establishes the surjectivity of $H^0(\mathcal O_X(D)) \to H^0(\mathcal O_{A_0+B_3}(K_\tau))$ for all even $\ell \geq 2$.

    \item[Step 2] If $\ell$ is odd.\par
    Similar to {\sffamily Case 1}, it suffices to prove for $\ell=3$. If $\tau$ is not flexible, then from the short exact sequence
    \[
        0 \to \mathcal O_X(K_\tau) \to \mathcal O_X( A_0+B_3+ K_\tau) \to \mathcal O_{A_0+B_3}(A_0+B_3 + K_\tau) \to 0,
    \]
    $h^1(K_\tau)=0$ implies that $H^0( \mathcal O_X(A_0+B_3+ K_\tau)) \to H^0(\mathcal O_{A_0+B_3}(A_0+B_3+ K_\tau))$ is surjective. Since $D = (A_0+B_3+ K_\tau) + 2(A_3+B_0)$, this shows that $H^0(\mathcal O_X(D)) \to H^0(\mathcal O_{A_0+B_3}(A_0+B_3+K_\tau))$ is surjective as desired. If $\tau = (10\ 00\ 00)$, then the linear system $\lvert D\rvert$ contains
    \[
        A_0 + 2A_3 + B_0 + B_1 + B_2 + C_1 + C_2.
    \]
    Note that it does not contain $B_3$. In particular, the map $H^0(\mathcal O_X(D)) \to H^0(\mathcal O_{A_0+B_3}(A_0+B_3+K_\tau))$ is nonzero and hence surjective, as $h^0(\mathcal O_{A_0+B_3}(A_0+B_3+K_\tau)) = 1$. For $\tau = (00\ 10\ 00)$, a similar argument applies; we have $4(A_3+B_0) + 2C_0 + A_3 + B_3 \in \lvert D \rvert$ and $h^0(\mathcal O_{A_0+B_3}(A_0+B_3+K_\tau)) = 1$. Now, let $\tau = (00\ 00\ 10)$. In this case, we have $h^0(\mathcal O_{A_0+B_3}(A_0+B_3+K_\tau)) = 2$. On the other hand, we may pick two members in $\lvert D \rvert$ as follows:
    \[
        2B_0 + B_1 + B_2 + B_3 + A_1 + A_2 + A_3,\ \ \text{and}\ \ 4(A_3+B_0) + A_0+B_0 + 2C_3. 
    \]
    The former contains $B_3$ but not $A_0$, while the latter contains $A_0$ but not $B_3$. This shows that the corresponding global sections of $\mathcal O_X(D)$ restrict to linearly independent sections of $\mathcal O_{A_0+B_3}(K_\tau)$. Consequently, the map $H^0(\mathcal O_X(D)) \to H^0(\mathcal O_{A_0+B_3}(K_\tau))$ is surjective.

\end{description}

    So far we have established the equality $h^0(D) = h^0(D- (A_0+B_3)) + h^0(\mathcal O_{A_0+B_3}(D))$ for $\ell \geq 2$. Thus, it is crucial to handle the case $\ell=1$, which we will do from now on.
\begin{description}
    \item[Step 3] Reduction to $\ell=1$. \par
        After the reduction to $\ell=1$, we rely on the short exact sequence
        \begin{equation}\label{eq: [2l;00l]+[KX] SES for l=1}
            0 \to \mathcal O_X(K_\tau) \to \mathcal O_X(A_0+B_3+K_\tau) \to \mathcal O_{A_0+B_3}(A_0+B_3+K_\tau) \to 0.
        \end{equation}
        This computation requires a case study, as the values $h^0(K_\tau)$, $h^0(\mathcal O_{A_0+B_3}(K_\tau))$, and $h^0(\mathcal O_{A_0+B_3}(A_0+B_3+K_\tau))$ vary over $\tau$.
    \item[Case 3-1] If $\tau = (10\ 00\ 00)$.\par
    In this case we have $h^0(\mathcal O_{A_0+B_3}(D)) = 1$ for every $\ell$, hence
    \[
        h^0(D) = h^0(D-(A_0+B_3)) + 1 = \dots = h^0(A_0+B_3+K_\tau) + (\ell-1).
    \]
    Moreover, $D_1 := 3(A_0+B_3) + C_0 + C_3 \in \lvert A_0+B_3+K_\tau \rvert$. Since $D_1$ does not contain $B_0$, we have $h^0(D_1) = h^0(D_1 -B_0)+1$. It can be easily checked that $h^0(D_1-B_0) = h^0(D_1-B_0-A_3) = h^0( K_X + \tau') = 1$ where $\tau' = (10\ 10\ 10)$. This shows that $h^0(A_0+B_3+K_\tau) = 2$. Consequently, $h^0(D) = \ell+1$.

    \item[Case 3-2] If $\tau = (00\ 10\ 00)$.\par
    In this case we have $h^0(\mathcal O_{A_0+B_3}(K_\tau)) = 2$ and $h^0(\mathcal O_{A_0+B_3}(A_0+B_3+K_\tau)) = 1$, thus
    \[
        h^0(D) = \left \{
            \begin{array}{ll}
                 h^0(D-A_0-B_3) + 2 & \text{if $\ell\geq 2$ is even} \\
                 h^0(D-A_0-B_3) + 1 & \text{if $\ell\geq 2$ is odd}.
            \end{array}
        \right.
    \]
    Combining, we get 
    \begin{equation}\label{eq: [2l;00l]+[KX] tau=(001000) case}
        h^0(D) = h^0(A_0+B_3 + K_\tau) + \left\lfloor \frac{3\ell-2}{2} \right\rfloor.
    \end{equation}
    Moreover, $D_2 := 2(A_3+B_0) + A_3 + B_3 + 2C_0 \in \lvert A_0+B_3 +K_\tau\rvert$ does not contain $A_0$, hence the map $H^0(\mathcal O_X(A_0+B_3+K_\tau)) \to H^0(\mathcal O_{A_0+B_3}(A_0+B_3+K_X))$ is surjective. This shows $h^0(A_0+B_3+K_\tau) = h^0(K_\tau) + h^0(\mathcal O_{A_0+B_3}(A_0+B_3+K_\tau)) = 3$, and thus, $h^0(D) = \lfloor \frac{3\ell+4}{2} \rfloor$.

    \item[Case 3-3] If $\tau = (00\ 00\ 10)$.\par
    In this case we have $h^0(\mathcal O_{A_0+B_3}(K_\tau)) = 1$ and $h^0(\mathcal O_{A_0+B_3}(A_0+B_3+K_\tau)) = 2$, thus
    \begin{equation}\label{eq: [2l;00l]+[KX] tau=(000010) case}
        h^0(D) = h^0(A_0+B_3+K_\tau) + \left\lfloor\frac{3\ell-3}{2} \right\rfloor.
    \end{equation}
    In the short exact sequence \eqref{eq: [2l;00l]+[KX] SES for l=1}, we have $h^1( \mathcal O_X(A_0+B_3+K_\tau)) \geq h^1(\mathcal O_{A_0+B_3}(A_0+B_3+K_\tau)) = 1$, thus $h^0(\mathcal O_X(A_0+B_3+K_\tau)) \geq 3$. We claim that the equality holds. Let $D_3 := A_0+B_3+K_\tau$. The divisor $A_0+B_3+C_0$ is nef and big by Proposition~\ref{prop: Kleiman criterion} and $(A_0+B_3+C_0)^2 > 0$. Thus, Kawamata-Viewheg vanishing theorem reads $h^1(D_3+C_0) = 0$. In particular, $h^0(D_3+C_0) = \chi(D_3+C_0) = 3$, showing that $h^0(D_3) \leq 3$. Consequently, we have $h^0(D) = \lfloor \frac{3\ell+3}{2} \rfloor$.
    
    \item[Case 3-4] If $\tau \neq (00\ 00\ 10)$ and $\tau = (00\ \epsilon0\ {*}\epsilon)$.\par
        In this case \eqref{eq: [2l;00l]+[KX] tau=(000010) case} is derived in the same way. Since $h^1(K_\tau)=0$, we have
        \[
            h^0(\mathcal O_X(A_0+B_3+K_\tau)) = h^0( \mathcal O_{A_0+B_3}(A_0+B_3+K_\tau)) + h^0(K_\tau) = 2 + h^0(K_\tau).
        \]
        Thus, $h^0(D) = \lfloor \frac{3\ell+3}{2} \rfloor$ if $\tau \neq (00\ 00\ 00)$, and $h^0(D) = \lfloor \frac{3\ell+1}{2}\rfloor$ if $\tau = (00\ 00\ 00)$.

    \item[Case 3-5] If $\tau \neq (00\ 10\ 00)$ and $\tau = (00\ \epsilon0\ {*}\bar \epsilon)$, where $\bar \epsilon = 1-\epsilon$.\par
        The equality \eqref{eq: [2l;00l]+[KX] tau=(001000) case} derived in the same way. Moreover, we have 
        \[
            h^0(A_0+B_3+K_\tau) = h^0(\mathcal O_{A_0+B_3}(A_0+B_3+K_\tau)) + h^0(K_\tau) = 2,
        \]
        thus $h^0(D) = \lfloor \frac{3\ell+2}{2}\rfloor$
    \item[Case 3-6] If $\tau$ is none of the above.\par
        In this case we have $h^0(K_\tau)=1$ and $h^0(\mathcal O_{A_0+B_3}(A_0+B_3+K_\tau))=1$, hence $h^0(A_0+B_3+K_\tau)=2$. In particular, $h^0(D) = h^0(A_0+B_3+K_\tau) + (\ell-1) = \ell+1$. \qedhere
\end{description}

\end{proof}

\bigskip
\section{Application to the study of Ulrich bundles on $X$}

Now, we use the results of previous sections to study Ulrich bundles on $X$.

\begin{proposition}\label{prop: numerical trivial class has e>1}
    Let $[D] = [d;a,b,c]$ be a numerical class with $d \geq 1$ and $\chi(D) = 0$. Then, $\e([D]) \geq 1$.
\end{proposition}
\begin{proof}
    Recall that $\ell := \frac 13 (d+a+b+c)$ is an integer. By \eqref{eq: 1-eff}, it suffices to prove that $M \leq \ell \leq d$, where $M = \max\{0,a,b,c\}$. By Riemann-Roch theorem, we have
    \begin{align*}
        2\chi(D) &= D^2 - (D.K_X) + 2  \\
        &= (\ell^2 - a^2 - b^2 - c^2 ) - (3\ell - a -b - c) + 2,
    \end{align*}
    which leads to the identity
    \begin{equation}\label{eq: Riemann-Roch for numerically trivial D}
        (2\ell-3)^2 + 2 = (2a-1)^2 + (2b-1)^2 + (2c-1)^2.
    \end{equation}
    Here, we invoke the Cauchy-Schwarz inequality:
    \[
        \left(a+b+c - \frac 32 \right)^2 \leq 3 \left(  \Bigl( a - \frac 12 \Bigr) ^2 + \Bigl( b-\frac 12\Bigr )^2 + \Bigl( c - \frac 12 \Bigr)^2 \right),
    \]
    which equivalently reads as
    \begin{equation}\label{eq: Cauchy-Schwarz for numerically trivial D}
        \Bigl(3\ell - d - \frac 32\Bigr )^2 \leq 3 \left(  \Bigl( \ell - \frac 32 \Bigr)^2 + \frac 12  \right).
    \end{equation}
    Solving the quadratic inequality \eqref{eq: Cauchy-Schwarz for numerically trivial D} in terms of $\ell$ yields
    \[
        \frac{1}{6} \left (  3d-9 - \sqrt{3} \sqrt{ (d-6)^2 + 3 } \right)  \leq \ell \leq \frac{1}{6} \left (  3d-9 + \sqrt{3} \sqrt{ (d-6)^2 + 3 } \right).
    \]
    One easily finds that for $d \geq 1$,
    \[
        \frac{1}{6} \left (  3d-9 - \sqrt{3} \sqrt{ (d-6)^2 + 3 } \right) \leq \frac 12 \left( 3d-9 + \sqrt 3 ( \lvert d-6 \rvert + 2 )  \right) < d,
    \]
    hence the inequality $\ell \leq d$ holds. The lower bound of $\ell$ can be negative for small values of $d$. For instance, if $d = 1$, then the left hand side of the previous inequality is $\frac 13 (-3 - \sqrt{21}) \approx -2.53$. Since the lower bound increases with $d$, we see that $\ell \geq -2$. Using the identity \eqref{eq: Riemann-Roch for numerically trivial D}, we will enumerate all the possible values of $a,b,c$ and use the inequality 
    \[
        5 \leq 2d+3 = 6\ell - (2a-1) - (2b-1) - (2c-1)
    \]
    to eliminate the cases $\ell \leq -1$. Suppose $\ell = -2$. Then $(2\ell-3)^2+2 = 51$. There are only few ways to express this number as the sum of three squares of odd numbers:
    \[
        51 = 7^2 + 1^2 + 1^2 = 5^2 + 5^2 + 1
    \]
    Since both $6\ell + 7 + 1 + 1 =-3$ and $6\ell + 5 + 5 + 1 = -1$ are smaller than $5$, we cannot have $d \geq 1$ when $\ell = -2$. A similar argument works for $\ell=-1$. Indeed, we have
    \[
        27 = 5^2 + 1^2 + 1^2 = 3^2 + 3^2 + 3^2,
    \]
    and $6\ell - (2a-1) - (2b-1) - (2c-1)$ is smaller than $5$. This proves that $\ell \geq 0$.
    
    For $\ell=0$, we have
    \[
        11 = 3^2 + 1^2 + 1^2,
    \]
    and $6\ell + 3 + 1 + 1 = 5$. Hence, up to symmetry, we have $[D] = [d;a,b,c] = [1;-1,0,0]$. This class contains $A_0$, hence $e([D]) \geq 1$. Now, suppose $\ell \geq 1$. From \eqref{eq: Riemann-Roch for numerically trivial D}, we have
    \[
        0 \leq \Bigl( \ell - \frac 32 \Bigr )^2 - \Bigl( a - \frac 12\Bigr)^2 - \frac 14 - \frac 14 + \frac 12,
    \]
    from which we read $(a-\frac 12)^2 \leq (\ell-\frac 32)^2$. If $\ell \geq 2$, then $\ell \geq a+1$. Similarly, $\ell \geq b+1$ and $\ell \geq c+1$ hold, thus $\ell \geq M$. This proves the inequality \eqref{eq: 1-eff} for $\ell \geq 2$. If $\ell = 1$, then $(a- \frac12)^2 \leq \frac 14$, hence either $a = 0$ or $a=1$. This proves $\ell \geq M$ for $\ell=1$.
\end{proof}

\begin{proposition}\label{prop: nefness of H-trivial D with d>=6}
    Assume that the divisor $D \in \Pic X$ is cohomologically trivial; that is, it satisfies $h^p(D) = 0$ for all $p$. If $d = (D.K_X) \geq 6$, then $D$ is nef.
\end{proposition}
\begin{proof}
    Assume $D$ is not nef. By Proposition~\ref{prop: Kleiman criterion}, there exists a curve $\sZ_i$ with $i\in\{0,3\}$ such that $(D.\sZ_i) < 0$. Without loss of generality, we may assume $(D.A_0) < 0$. In the short exact sequence
    \[
        0 \to \mathcal O_X(D-A_0) \to \mathcal O_X(D) \to \mathcal O_{A_0}(D) \to 0,
    \]
    we have $h^1(\mathcal O_{A_0}(D)) > 0$ since $A_0$ is an elliptic curve and $(D.A_0) < 0$. Then,
    \[
        h^1(\mathcal O_{A_0}(D)) = h^2(D - A_0) = h^0(A_0+K_X - D) > 0.
    \]
    Since $({A_0 + K_X - D} \mathbin. K_X) \geq 0$, we have $d \leq 7$. If $d=7$, then $D = A_0+K_X$. This contradicts $h^0(D) = 0$, hence $D$ is nef. If $d = 6$, $A_0+K_X-D = \sZ_i$ for suitable $i \in \{0,3\}$ and $\sZ \in\{ A,B, C\}$. Then, $D = A_0 + K_X - \sZ_i$ is not nef, hence $\sZ_i$ is either $B_3$ or $C_3$. For both cases, one can directly check that $h^1(D) = 1$, thus $D$ is not cohomologically trivial.
\end{proof}
\begin{corollary}\label{cor: H-trivial divisors with d>=7}
    If $D \in \Pic X$ is cohomologically trivial and $d = (D.K_X) \geq 7$, then up to symmetry, $D$ belongs to one of the following numerical classes:
    \begin{enumerate}
        \item $[2\ell+1;\, 0,0,\ell-1]$;
        \item $[2\ell;\, 0, 1, \ell-1]$;
        \item $[2\ell;\, 0, \ell-1, 1]$.
    \end{enumerate}
\end{corollary}
\begin{proof}
    Since $D \in [D]$ is not effective, $\e([D]) \neq 64$. By Proposition~\ref{prop: numerical trivial class has e>1} and Proposition~\ref{prop: nefness of H-trivial D with d>=6}, $D$ is nef and $\e([D]) \geq 1$. Hence, the result essentially follows from Proposition~\ref{prop: exotic classes (nef but not 64-eff)}, while the case $[2\ell;\, 0,0,\ell]$ is excluded as its holomorphic Euler characteristic is $1 - \ell \neq 0$.
\end{proof}
\begin{proposition}\label{prop: ample and bpf divisor}
    Let $H \in \Pic X$ be an ample and base point free divisor. If $H \in [ h;\, a(H), b(H), c(H)]$, then $h \geq 12$ and $a(H), b(H), c(H) \geq 2$.
\end{proposition}
\begin{proof}
    By Kleiman's criterion, $(H.C) > 0 $ for any curve $C \subset X$. Thus, $a(H) = (H . A_0) > 0$. Assume $a(H) = 1$. Since $H$ is base point free, there exists $D \in \lvert H \rvert$ whose support does not contain $A_0$. Then, $D$ intersects $A_0$ precisely at one point, say $P$. If the $\Pic A_0$-part of the symmetric coordinates of $D$ is $(1\ \epsilon_1\epsilon_2)$, then $P$ and $A_0^{\epsilon_1\epsilon_2}$\,(see Section~\ref{sec: Burniat config} for the notation) are linearly equivalent. Since the curve $A_0$ is not rational, we have $P = A_0^{\epsilon_1 \epsilon_2}$. So far we have shown the following: every member of $\lvert H \rvert $ contains either $A_0$ or $A_0^{\epsilon_1 \epsilon_2}$ in its support. This proves that $A_0^{\epsilon_1 \epsilon_2}$ is a base point of $\lvert H\rvert$, a contradiction. It follows that $a(H) \geq 2$. Similarly, $b(H),\,c(H) \geq 2$.
    
    Moreover, we also have
    \[
        (H . A_3),\, (H.B_3) , \, (H. C_3) \geq 2.
    \]
    To describe the above numbers in terms of $h, a(H), b(H), c(H)$, we use the following divisor in $[H]$:
    \[
        \ell(A_3 + B_0 + C_0) - a(H) A_0 - b(H) B_0 - c(H) C_0, \text{ where }\ell = \frac 13 (h + a(H) + b(H) + c(H)).
    \]
    Computing intersection numbers, we obtain $(H.A_3) = \ell - b(H) - c(H)$, $(H.B_3) = \ell - a(H) - c(H)$, and $(H.C_3) = \ell - a(H)-b(H)$. Then,
    \begin{align*}
        6 &\leq (H.A_3) + (H.B_3) + (H.C_3) \\
        &= 3\ell - 2(a(H) + b(H) + c(H)) \\
        &= h - a(H) + b(H) + c(H).
    \end{align*}
    Thus, $h \geq 6 + a(H) + b(H) + c(H) \geq 12$.
\end{proof}

\begin{theorem}\label{thm: no Ulrich line bundles}
    Let $H \in \Pic X$ be an ample and base point free divisor. Then, there does not exist a divisor $D \in \Pic X$ such that both $D$ and $D-H$ are cohomologically trivial.
\end{theorem}
\begin{proof}
    Assume both $D$ and $D-H$ are cohomologically trivial. For a divisor $E \in \Pic X$, let $[E] = [d(E);\, a(E), b(E), c(E)]$, and let $d = d(D)$, $a=a(D)$, $b=b(D)$, $c=c(D)$, and $h = d(H)$. By Proposition~\ref{prop: ample and bpf divisor}, we have $h \geq 12$, hence either $d-h = d(D-H) \leq -1$ or $d \geq 7$. By Serre duality, both $K_X-D$ and $K_X-D+H$ are also cohomologically trivial. Replacing $D$ by $K_X-D+H$ if necessary, we may assume $d \geq 7$. By Corollary~\ref{cor: H-trivial divisors with d>=7}, $[D]$ is symmetric to either $[2\ell+1;\, 0,0,\ell-1]$, $[2\ell;\, 0,1,\ell-1]$, or $[2\ell;\,0,\ell-1,1]$. Then $D$ is strictly nef, hence $D-H$ is not nef. By Proposition~\ref{prop: nefness of H-trivial D with d>=6}, we have $d(D-H) \leq 5$.
    
    Suppose $d(D-H) \leq -1$. By Serre duality, $K_X - D + H$ is cohomologically trivial and $d( K_X - D + H ) \geq 7$. Thus, by the previous observation, $[K_X - D + H]$ is symmetric to either $[2\ell'+1;\, 0,0,\ell'-1]$, $[2\ell';\, 0,1,\ell'-1]$, or $[2\ell';\,0,\ell'-1,1]$. On the other hand, $a(K_X-D+H) = 1- a + a(H) \geq 0$, hence $a \geq a(H)+1 \geq 3$ by Proposition~\ref{prop: ample and bpf divisor}. Similarly, we have $b,c \geq 3$, which contradicts Corollary~\ref{cor: H-trivial divisors with d>=7}.
        
    It remains to consider $0 \leq d(D-H) \leq 5$. For notational convenience, let us write $[D-H] = [d' ;\, a',b',c']$ and $\ell' = \frac 13 (d'+a'+b'+c')$. By Corollary~\ref{cor: H-trivial divisors with d>=7} and Proposition~\ref{prop: ample and bpf divisor}, at least two of $\{a',b',c'\}$ are negative, and $\min\{a',b',c' \} \leq -2$. From the identity \eqref{eq: Riemann-Roch for numerically trivial D} we read
    \[
        (2 \ell' - 3)^2 + 2 \geq 5^2 + 3^2 + 1^2 = 35,
    \]
    thus $\lvert 2\ell'-3 \rvert \geq 7$. Moreover, by Proposition~\ref{prop: numerical trivial class has e>1} and \eqref{eq: 1-eff}, $0 \leq \ell' \leq d' \leq 5$, hence we must have $\ell'=d'=5$. By \eqref{eq: Riemann-Roch for numerically trivial D},
    \[
        (2a'-1)^2 + (2b'-1)^2 + (2c'-1)^2 = 51.
    \]
    One finds that the only possible combination is $51 = 5^2 + 5^2 + 1^2$. Up to permutation, we have $a'=b'=-2$ and $c' \in \{0,1\}$. This leads to a contradiction as $d' = 3\ell' - a' -b' - c' \geq 18$.
\end{proof}

We finish the paper by presenting an Ulrich bundle of rank $2$ over $(X, 3K_X)$.
\begin{theorem}\label{thm: Ulrich of rank 2}
    Let $H = 3K_X$, and let $D_1$ be the divisor
    \[
        \begin{array}{ c | r r r | r r r}
            \Bigl( 10 & \bigl(0\ 01) & \bigl(1\ 11) & \bigl(4\ 01) & (0\ 01) & (1\ 11) & (4\ 11) \Bigr)
        \end{array},
    \]
    and let $D_2 = 4K_X - D_1$. There exists an Ulrich bundle $\mathcal E$ of rank $2$ which fits into the sequence
    \begin{equation}\label{eq: Ulrich of rank 2}
        0 \to \mathcal O_X(D_1) \to \mathcal E(-H) \to \mathcal O_X(D_2) \otimes \mathcal I_Z \to 0,
    \end{equation}
    where $Z\subset X$ is a subscheme of $6$ points in general position.
\end{theorem}
\begin{proof}
    The proof is a standard application of Serre correspondence\,(see~\cite[Section~5.1]{HuybrechtsLehn}). To see that $\mathcal E$ is locally free, we need to check the Cayley-Bacharach property for the pair $(K_X+D_2-D_1,\, Z)$: namely, if $s \in H^0(K_X+D_2-D_1)$ vanishes along a subscheme $Z' \subset Z$ of length $5$, then $s$ vanishes along $Z$. As $h^0(K_X+D_2-D_1) = 1$ and $Z$ is in general position, such $s$ with $s \big\vert_{Z'} = 0$ should be the zero section. Hence, the Cayley-Bacharach property holds trivially for the pair $(K_X+D_2-D_1,\, Z)$, and $\mathcal E$ is a locally free sheaf of rank $2$. Using the methods developed in Section~\ref{sec: Main Algorithm}, we find that $D_1$ is cohomologically trivial, and that $h^0(D_2)=6$, $h^q(D_2)=0$ for $q=1,2$. From the short exact sequence
    \[
        0 \to \mathcal O_X(D_2) \otimes \mathcal I_Z \to \mathcal O_X(D_2) \to \mathcal O_Z \to 0,
    \]
    we have $H^0(\mathcal O_X(D_2)) \simeq H^0(\mathcal O_Z)$ as $Z$ is the subscheme of $6$ points in general position. In particular, we have $h^p(\mathcal O_X(D_2) \otimes \mathcal I_Z) = 0 $ for $p=0,1,2$. Then by \eqref{eq: Ulrich of rank 2}, $\mathcal E(-H)$ is cohomologically trivial. Since $\det 
\mathcal E(-H) = \mathcal O_X(D_1 + D_2) = \mathcal O_X(4K_X)$, we have $\mathcal E \simeq \mathcal E^\vee \otimes \det \mathcal E = \mathcal E^\vee(10K_X)$. By Serre duality, we have
    \[
        h^p( \mathcal E(-2H)) = h^{2-p}( \mathcal E^\vee ( 2H + K_X) ) = h^{2-p}( \mathcal E(-H)), 
    \]
    hence $\mathcal E(-2H)$ is cohomologically trivial. This proves that $\mathcal E$ is an Ulrich bundle.
\end{proof}

\begin{remark}\label{rmk: why differ from Casnati's}
    The existence of Ulrich bundles of rank $2$ over $(X, 3K_X)$ is already confirmed in the work \cite{Casnati:SpecialUlrich} by Casnati. In \cite[Theorem~1.1]{Casnati:SpecialUlrich}, Casnati proves that there exists an Ulrich bundle $\mathcal F$ of rank $2$ which fits into the short exact sequence
    \[
        0 \to \mathcal O_X(H+K_X) \to \mathcal F \to \mathcal O_X(2H) \otimes \mathcal I_W \to 0,
    \]
    where $W$ is a general set of $20$ points. From the construction, we have $h^0(\mathcal F(-H-K_X)) \neq 0$.
    
    In Theorem~\ref{thm: Ulrich of rank 2}, however, we have $h^0(D_1-K_X) = h^0(D_2-K_X) = 0$, thus $h^0(\mathcal E(-H-K_X))=0$. It follows that Theorem~\ref{thm: Ulrich of rank 2} presents an Ulrich bundle which cannot be obtained by the method of Casnati. It would be interesting to ask whether these two bundles lying in the same irreducible component of the moduli space of Ulrich bundles. 
\end{remark}
%

\end{document}